\documentclass[11pt]{amsart}
\usepackage[utf8]{inputenc}

\usepackage{amsmath,amssymb,amsfonts,amsthm}
\usepackage{graphicx}
\usepackage{amsthm}
\usepackage[colorlinks=true]{hyperref}
\usepackage{epsfig, enumerate}
\usepackage{tikz}

\numberwithin{equation}{section}
\newcommand{\R}{\mathbb{R}}
\newcommand{\N}{\mathbb{N}}
\newcommand{\Z}{\mathbb{Z}}

\newcommand{\beqnn}{\begin{eqnarray*}}
\newcommand{\eeqnn}{\end{eqnarray*}}
\newtheorem{thm}{Theorem}[section]
\newtheorem{lem}[thm]{Lemma}
\newtheorem{prop}[thm]{Proposition}
\newtheorem{rmk}[thm]{Remark}
\newtheorem{defi}{Definition}[section]

\begin{document}

\title{Complexity of Hofer's geometry in higher dimensional manifolds}

\author{Zhijing Wendy Wang}\address{Department of Mathematics, University of Chicago, Chicago, IL 60637, USA}\email{zhijingw@uchicago.edu}

\begin{abstract}
This paper establishes robust obstructions to representing Hamiltonian diffeomorphisms as $k$-th powers ($k \geq 2$) or embedding them in flows for certain higher-dimensional symplectic manifolds $(M,\omega)$, including surface bundles. We prove that in the Hamiltonian group $(\mathrm{Ham}(M,\omega), d_H)$ equipped with the Hofer metric, there exist arbitrarily large balls that are disjoint from the set of $k$-th powers. Furthermore, we demonstrate that the free group on two generators embeds into every asymptotic cone of $(\mathrm{Ham}(M,\omega), d_H)$, revealing the large-scale geometric complexity of the Hamiltonian group.
\end{abstract}
\maketitle

\tableofcontents

\section{Introduction}

\subsection{Background and main result}
Let $(M,\omega)$ be a symplectic manifold. Given a compactly supported smooth function $H: S^1 \times M \to \mathbb{R}$, called a \emph{Hamiltonian}, there exists a unique time-dependent vector field $X_t$ such that $\omega(X_t, \cdot) = -dH_t$. Let $\phi^t_H$ be the flow generated by $X_t$. The time-$1$ map $\phi^1_H$ is called a Hamiltonian diffeomorphism. We denote by $\mathrm{Ham}(M,\omega)$ the set of all such Hamiltonian diffeomorphisms. A direct calculation shows that $\mathrm{Ham}(M,\omega)$ is a subgroup of the group of symplectomorphisms, which we denote by $\mathrm{Symp}(M,\omega) = \{\phi \in \mathrm{Diff}(M) : \phi^*\omega = \omega\}$. 

Hamiltonian diffeomorphisms form a fundamental class of transformations in symplectic geometry, arising naturally as time-1 maps of Hamiltonian flows generated by smooth functions on symplectic manifolds. These maps play a central role in classical mechanics, symplectic topology, and dynamical systems, encoding rich geometric and dynamical information. Understanding the structure and geometry of the group \(\mathrm{Ham}(M,\omega)\) of Hamiltonian diffeomorphisms, equipped with the Hofer metric, is a major theme in modern symplectic geometry.

The Hofer metric, introduced by Hofer \cite{Hofer}, provides a bi-invariant, non-degenerate metric on \(\mathrm{Ham}(M,\omega)\), measuring the minimal oscillation of Hamiltonians generating a given diffeomorphism. For $\phi \in \mathrm{Ham}(M,\omega)$, its Hofer norm is defined as
\[
\|\phi\|_H = \inf_{H : \phi^1_H = \phi} \int_0^1 \left( \max_M H_t - \min_M H_t \right) dt.
\]
This is a non-degenerate norm on all symplectic manifolds (see \cite{Hofer}), which induces the Hofer metric $d_H(\phi, \psi) = \|\phi \circ \psi^{-1}\|_H$. The Hofer metric is bi-invariant, since the Hofer norm is conjugation invariant by its definition.

This metric captures the "size" and complexity of Hamiltonian diffeomorphisms and has led to deep insights into symplectic rigidity and flexibility phenomena.

Within \(\mathrm{Ham}(M,\omega)\), two natural subsets arise: the autonomous Hamiltonian diffeomorphisms \(\mathrm{Aut}(M,\omega)\),  which are generated by time-independent Hamiltonian functions $H: M \to \mathbb{R}$; and the set of \(k^{\text{th}}\) powers $\mathrm{Ham}^k(M,\omega) = \{ f \in \mathrm{Ham}(M,\omega) : \exists\, g \in \mathrm{Ham}(M,\omega) \text{ such that } f = g^k \}$, consisting of diffeomorphisms that can be written as the \(k\)-fold composition of another Hamiltonian diffeomorphism. Autonomous maps correspond to simpler, time-independent dynamics, and powers reflect algebraic structure within the group. By definition, an autonomous Hamiltonian diffeomorphism is the time-$1$ map $\phi_H^1$ for a time-independent $H$ and can therefore be written as a power of another Hamiltonian diffeomorphism $\phi_H^{\frac{1}{k}}$. So we always have the inclusions $\mathrm{Aut}(M,\omega) \subset \mathrm{Ham}^k(M,\omega) \subset \mathrm{Ham}(M,\omega)$ for any symplectic manifold $M$ and any $k$.  

A fundamental question is: How large are these subsets within \(\mathrm{Ham}(M,\omega)\) when measured by the Hofer metric? In other words, can every Hamiltonian diffeomorphism be approximated by autonomous maps or powers, or are these subsets ``small" in a geometric sense?

Polterovich and Shelukhin \cite{PS} answered this question for surfaces $\Sigma_g $ of genus $g\ge 4$, showing that the complement of autonomous maps contains arbitrarily large Hofer balls by constructing a family of Hamiltonian diffeomorphisms $\{f_n\}$ such that $d_H(f_n,\;\mathrm{Aut}(\Sigma_g,\omega))\to \infty$ as $n\to \infty$. This reveals that genuinely time-dependent Hamiltonian diffeomorphisms form a large and geometrically significant part of the group. The same result was extended to product manifolds $\Sigma_g\times M^\prime$ in \cite{PS},\cite{PSS} and \cite{Zhang}; and to surfaces $\Sigma_g$ of genus $g=2,3$ by \cite{Chor}. 

This work extends this phenomenon to a broader class of higher-dimensional symplectic manifolds, revealing new layers of complexity in their Hamiltonian groups. We shall work on manifolds with very large first homotopy group.

\begin{defi}\label{cond*}
    Let $(M,\omega)$ be a symplectic manifold. We say that $(M,\omega)$ \emph{satisfies condition  (*)}, if \begin{itemize}
        \item it is either a closed, symplectically aspherical manifold  or a Weinstein domain;
        
        \item it contains two Lagrangian tori $T_1$ and $T_2$ embedded in $M$, transversely intersecting at one point;
       
        \item  the natural homomorphism $\Z^n *\Z^n\simeq \pi_1(T_1\cup T_2)\to \pi_1(M)$ is monomorphic;
        
        \item there exists two homotopically non-trivial curves $a\subset T_1$ and $b\subset T_2$ such that $M$ is $\alpha$-atoroidal for any $\alpha\in 
        \langle a,b\rangle$ (meaning that the symplectic form and first Chern class vanish on any tori swept out by $\alpha$, see Section \ref{floer}).
    \end{itemize} 
\end{defi}

Examples of manifolds satisfying these conditions include the following:

(1) $M=T^*\mathbb{T}^n \cup_\psi T^*\mathbb{T}^n$ is the plumbing of cotangent bundles of two $n$-dimensional tori, equipped with the standard symplectic form.

(2) $M=S_1\ltimes_\psi S_2$ is a surface bundle, where $S_1,S_2$ are closed surfaces of genus $\ge 2$ with appropriate monodromy $\psi: \pi_1(S_1)\to Symp(S_2)$, equipped with the symplectic form $\omega=\omega_1+k\omega_2$ where $\omega_1,\omega_2$ are the volume forms on $S_1$ and $S_2$ and $k>0$. Suppose that there exist simple closed curves $a_0,b_0\subset S_1 ;\; c,d\subset S_2$, $|a_0\pitchfork b_0|=|c\pitchfork d|=1$, such that $\psi(a_0)(c)=c$, $\psi(b_0)(d)=d$. Then $T_1$ generated by $a_0,c$ and $T_2$ generated by $b_0,d$ are two embedded Lagrangian tori in $M$, transversely intersecting at one point. Under many choices of $\psi$, the tori $T_1$ and $T_2$ will generate a free product in $\pi_1(M)$. We shall  provide some detailed constructions in Section \ref{surfbund}.

(3) Luttinger surgery on examples in (2). Let $M$ be a symplectic $4$-manifold, and $T$ be a Lagrangian torus in $M$ with a homotopically nontrivial closed curve $\gamma\subset T$, then the Luttinger surgery of $M$, denoted by $M^\prime= M(T,\gamma,k)$ is the manifold obtained from cutting out a tubular neighborhood of $T$ and gluing back a copy of $U_r=\mathbb T^2\times [-r,r]\times [-r,r]$, using a symplectomorphism $\psi_k$, which is topologically a $1/k$ Dehn surgery along $\gamma$. (See Section \ref{lutsurg} for more details.) Let $T_1$ and $T_2$ be the Lagrangian tori in $M=S_1\ltimes_\psi S_2$ given in (2), and let $T$ be a Lagrangian tori in $M$ disjoint from $T_1$ and $T_2$, then for appropriate choices of $T,\gamma$, the resulting manifold $M(T,\gamma,k)$ also satisfies condition (*).

(4) Symplectic sums based on examples in (2). We take a surface bundle $M_1=S_1\ltimes_{\psi_1}S_2$ as in example (2), let $N_1\simeq S_2$ be a fiber disjoint from $T_1$ and $T_2$. Take another surface bundle $M_2=S_2\ltimes_{\psi_2}S_3$ that has a section $N_2\simeq S_2$, with Euler class $0$. Then we may take a symplectic sum $M_1\#_{N_1\simeq N_2}M_2$ and it still satisfies condition (*). See Section \ref{sympsum} for a more detailed discussion.

Our first main theorem shows that for any integer \(k \geq 2\), there exist Hamiltonian diffeomorphisms that are arbitrarily far in the Hofer metric from the set \(\mathrm{Ham}^k(M,\omega)\) of \(k^{\text{th}}\) powers. In particular, they are arbitrarily far from autonomous maps.

%\begin{defi} Let $f:M\to M$ be a continuous map on a smooth manifold $M$, we say that $f$ is $m$-dimensional transitive, if $m=\sup\{\mathrm{Hdim} (\overline{\{f^n(x):n\in \Z\}}):x\in M\}$. \end{defi}

\begin{thm}\label{nokthroot}
    Let $(M,\omega)$ be a symplectic manifold of dimension $2n$, satisfying condition (*). Then for each integer $k \geq 2$, there exists a sequence $\{f_n\}_{n=1}^\infty \subset \mathrm{Ham}(M)$ such that $$d_H(f_n, \mathrm{Ham}^k(M,\omega)) \to \infty$$ as $n \to \infty$. In particular, $$d_H(f_n, \mathrm{Aut}(M,\omega)) \to \infty$$ as $n \to \infty$.
\end{thm}

This result quantifies the geometric``smallness" of powers and autonomous maps, showing that some Hamiltonian diffeomorphisms are genuinely and robustly time-dependent and cannot be approximated by powers. This has important implications for the dynamical complexity and algebraic structure of \(\mathrm{Ham}(M,\omega)\).

A main difference between our work and the aforementioned results is that our construction is genuinely higher dimensional. In \cite{PS}, \cite{PSS} and \cite{Zhang}, for the product manifolds $\Sigma_g\times M$, the Hamiltonian diffeomorphisms $\{f_n\}$ such that $d_H(f_n,\mathrm{Aut}(\Sigma_g,\omega))\to \infty$ were constructed as the product map $f_n=\{\phi_n\times id\}$ on $\Sigma_g\times M$ where $\phi_n\in \mathrm{Ham}^k(\Sigma_g,\omega)$. Consequently, every periodic orbit of such an $f_n$ is confined to a two-dimensional submanifold $\Sigma_g \times \{\text{pt}\}$.

By contrast, the manifolds we consider are not products. In the case of a surface bundle $M$, the Hamiltonian diffeomorphisms we construct are not fiber-preserving and do not project to the base. More significantly, the dynamical behavior of our sequence $\{f_n\}$ from Theorem \ref{nokthroot} is richer: we obtain periodic orbits that become arbitrarily dense in an open set as $n$ increases.

\begin{prop}\label{prop:denseorbits}
    Let $\{f_n\}$ be the sequence of Hamiltonian diffeomorphisms constructed in Theorem \ref{nokthroot}. For $n$ sufficiently large, the number of $k$-periodic points of $f_n$ grows at least exponentially in $k$. Furthermore, there exists an open set $U \subset M$ and a sequence of periodic points $\{x_n\}$ of $f_n$ such that the orbits $\mathcal{O}(f_n, x_n) = \{f_n^k(x_n) : k \in \mathbb{N}\}$ become dense in $U$ as $n \to \infty$. That is, for every open subset $V \subset U$, there exists $N \in \mathbb{N}$ such that for all $n > N$, we have $\mathcal{O}(f_n, x_n) \cap V \neq \varnothing$.
\end{prop}

Beyond local geometry, we also investigate the large-scale or coarse geometry of \(\mathrm{Ham}(M,\omega)\) via its asymptotic cone, a tool introduced by Gromov \cite{Gromov93} to study metric spaces from a "zoomed-out" perspective. The asymptotic cone captures the group's behavior at infinity and inherits a group structure when the metric is bi-invariant.

Informally, the asymptotic cone of a metric space $(Y, d)$ with a base point $x_0$ is obtained by taking a rescaling limit, viewing the space from increasingly distant points. More precisely, fix a non-principal ultrafilter $\mathcal{U}$ on $\N$. The asymptotic cone $\mathrm{Cone}_\mathcal{U}(Y, d)$ is defined as the set of equivalence classes of sequences $(x_j)_{j=1}^\infty \subset Y$ satisfying $\limsup_{j \to \infty} \frac{d(x_j, x_0)}{j} < \infty$, where two sequences are equivalent, $(x_j) \sim (y_j)$, if and only if $\lim_{\mathcal{U}} \frac{d(x_j, y_j)}{j} = 0$. The metric on the cone is given by $d_\mathcal{U}([(x_j)], [(y_j)]) = \lim_{\mathcal{U}} \frac{d(x_j, y_j)}{j}$.

If $Y = G$ is a group equipped with a bi-invariant metric, then the asymptotic cone $\mathrm{Cone}_\mathcal{U}(G, d)$ inherits a group structure (see Proposition 3.3 of \cite{CZ11}), with the group operation defined entrywise on representative sequences.

Our second main theorem establishes that the free group of two generators \(\mathbb{F}_2\) embeds faithfully into the asymptotic cone of \(\mathrm{Ham}(M,\omega)\) for manifolds satisfying condition (*).

\begin{thm}\label{F2incone}
    Let $(M,\omega)$ be a symplectic manifold satisfying condition (*). Then for any non-principal ultrafilter $\mathcal{U}$ on $\N$, there exists a faithful homomorphism
    \[
    \mathbb{F}_2 \to \mathrm{Cone}_{\mathcal{U}}(\mathrm{Ham}(M,\omega), d_H).
    \]
    That is, the free group of two generators embeds into the asymptotic cone.
\end{thm}

This reveals a highly non-abelian and rich large-scale structure of the Hamiltonian group. There are known results on embedding abelian groups into Hamiltonian group, while the embedding of free groups is a more recent development. In particular, work by D. Álvarez-Gavela et al. \cite{Junzhangetal} demonstrated an embedding of $\mathbb{F}_2$ into the Hamiltonian group of certain higher-genus surfaces and their products. Chor \cite{Chor} further extended their result to surfaces of genus $2$ and $3$. Our result contributes to this line of inquiry by establishing such embeddings in the asymptotic cone for a new class of higher-dimensional symplectic manifolds.

The conditions on the manifold in Theorems \ref{nokthroot} and \ref{F2incone} can be generalized in the following ways.

\begin{rmk}\label{cond**}
    In fact, for the purposes of Theorems \ref{nokthroot} and \ref{F2incone}, the last two items in condition (*) can be replaced with the following condition.
    
     \begin{itemize}

        \item   There exists curves $a\subset T_1$ and $b\subset T_2$, such that if $a^{k_1}b^{l_1}\cdots a^{k_m}b^{l_m}=g_1h_1\cdots g_mh_m$, where $k_i,l_j\in \Z\backslash\{ 0\}$ and $g_i\in \pi_1(T_1),h_j\in \pi_1(T_2)$ for any $1\le i,j\le m$, then $g_i=a^{k_i}$ and $h_j=b^{l_j}$ for any $1\le i,j\le m$, we shall say that any word in $\langle a,b\rangle$ is in unique reduced form in $\langle \pi_1(T_1), \pi_1(T_2)\rangle$.
        
        \item  $M$ is $\alpha$-atoroidal for any $\alpha\in 
        \langle a,b\rangle$.
    \end{itemize} 

   We say that a manifold $M$ satisfies condition (**), if it satisfies these conditions above and the first two items in condition (*). It is clear from definition that condition (*) implies condition (**).
\end{rmk}

\begin{rmk}
    The conclusions of Theorems \ref{nokthroot} and \ref{F2incone} also apply to product manifolds $M = M_1 \times M_2$, where $M_1$ is a closed manifold that satisfies condition (**) and $M_2$ is any closed symplectically aspherical manifold. 
    
    Indeed, the Hamiltonian diffeomorphism $\tau(N,w)$ can be defined on $M_1$ as in Section \ref{ltm} and then extended as $\tau(N,w) \times \mathrm{id}_{M_2}$ to the product. The necessary Floer–theoretic invariants for the product map are determined by those on $M_1$ via the Künneth formula for filtered Floer homology; consequently, the estimates on action differences and boundary depths persist. Thus, the same proofs yield analogous results for $M_1 \times M_2$.
\end{rmk}

%Some further questions include, whether the $f_n$ we constructed are transitive on its support and whether the results can be extended to symplectic manifolds that contain higher genus/hyperbolic manifold. 

\subsection{Organization and outline of the proof}\label{outl}

The core of our argument is the construction of a sequence of Hamiltonian diffeomorphisms $\{f_n\}$ that diverges from the set of $k^{\text{th}}$ powers, $\mathrm{Ham}^k(M, \omega)$. This is achieved by analyzing the Floer-theoretic properties of a family of linked twist maps on $M$, which is akin to the Hamiltonian eggbeater maps studied in \cite{PS} and \cite{Junzhangetal}.

In a manifold satisfying condition (*), we may identify a neighborhood of $T_1\cup T_2$ with a neighborhood of the zero section of the plumbing domain $P_\psi(T_1,T_2)$. Our specific construction on $P_\psi(T_1,T_2)$ is as follows.  Let $\tau_1$ and $\tau_2$ be Hamiltonian diffeomorphisms that, near the zero section of their respective tori, coincide with a time-1 map of a geodesic flow, and are the identity away from it.  We then define Hamiltonian diffeomorphisms $f_N$ and $\tau(N,w)$ for any word $w$ as a finite composition of twists $\tau_1^N$ and $\tau_2^N$, see Equation \ref{taunw}.

The proof proceeds in three main steps:
\begin{itemize}
    \item Periodic Orbit Analysis: We identify specific free homotopy classes $\gamma_N$ in which $f_N$ has only finitely many fixed points and estimate the action functional of these periodic orbits. The purpose of condition (*) is to make sure the periodic points of $f_N$ are in different homotopy classes.

    \item Index and Boundary Depth Calculation: We compute the Conley-Zehnder indices of these periodic orbits. The relationship between the index and the action, combined with the structure of the Floer differential, allows us to prove a lower bound $\delta N > 0$ on the action difference $\mathcal{A}_{H}(q) - \mathcal{A}_{H}(p)$ for certain $p\in M$ and any $q\in \partial^{-1}p$. This implies a lower bound $\delta N$ on the length of bars.
    
    \item Applying the relation between action difference and Hofer norm, given in Propositions \ref{bounddepth} and \ref{dishamk}, we force the Hofer distance from $f_N$ to any $k^{\text{th}}$ power to go to infinity which proves Theorem \ref{nokthroot}. We also force linear growth of the Hofer norm of $\tau(N,w)$, which shows that the sequence $\tau(N,w)$ is a non-trivial element of the asymptotic cone for any $w\in \mathbb F_2$, which proves Theorem \ref{F2incone}.
\end{itemize}

The paper is organized to reflect this strategy. In Section \ref{cloconst}, we give some constructions of closed symplectic manifolds satisfying condition (*). In Section \ref{prelim}, we construct the surface bundle examples and recall the necessary Floer theory, focusing on the action functional, filtered homology, and the relation between action differences and the Hofer norm.
 In Section \ref{ltm}, we construct the Hamiltonian diffeomorphisms $\tau(N,w)$, analyze their periodic points, and estimate their actions.
 In Section \ref{cz}, we compute the Conley-Zehnder indices of these periodic points and prove the required lower bound on the action difference.

\subsection{Acknowledgement}
The author is deeply grateful to Leonid Polterovich for his invaluable guidance throughout this work. His introduction to these problems and our many productive discussions were essential to this research. The author also thanks Robert Gompf for very helpful communications that led to the examples given by Luttinger surgery and symplectic sum, and thanks Jinxin Xue and Egor Shelukhin for very useful conversations and comments.

\section{Constructions of closed examples}\label{cloconst}

In this section, we provide more details for constructing manifolds that satisfies condition (*) and (**). We shall construct some surface bundle examples and use Luttinger surgery and symplectic sum to give more examples.

\subsection{Surface bundle examples}\label{surfbund}

  We consider a surface bundle $M=S_1\ltimes_\psi S_2$ where $\psi:\pi_1(S_1)\to Symp(S_2)$ is the monodromy group. We take the symplectic form $\omega=\omega_1+k\omega_2$ in $M$ where $\omega_1,\omega_2$ are the volume forms on $S_1$ and $S_2$ and $k\in \R_+$. Suppose that the base and fiber are surfaces of genus $\ge 2$.

 %We first show that $M$ is symplectically aspherical and atoroidal. 
  
  Take $\psi$ such that there exist simple closed curves $a_0,b_0\in \pi_1(S_1) , c,d\in \pi_1(S_2)$, $|a_0\pitchfork b_0|=|c\pitchfork d|=1$, with $\psi(a_0)c=c$, $\psi(b_0)d=d$. Then $T_1$ generated by $a_0,c$ and $T_2$ generated by $b_0,d$ are two embedded Lagrangian tori in $M$, transversely intersecting at one point. We provide two constructions so that $M$ satisfies condition (*).
  
  For example, we may take $a_0,b_0\subset S_1$ that generate a free group in $\pi_1(S_1)$ and $c,d,e,f\subset S_2$ such that $\langle c,d,e\rangle , \langle c,d,f\rangle$ generate free subgroups of in $\pi_1(S_2)$ respectively, where $|c\cap f|=0,|d\cap e|=0$. Take $\psi$ such that $\psi(a_0)=\tau_e,\psi(b_0)=\tau_f $, where $\tau_\gamma$ is the Dehn twist around $\gamma$. Then the tori $T_1$ generated by $a_0,c$ and $T_2$ generated by $b_0,d$ generate a free subgroup $(\Z^2)*(\Z^2)\subset \pi_1(M)$ in $M$ and condition (*) is naturally satisfied.
  
  A more general example is to take $\psi(a_0)$ as pseudo-Anosov on $S_2-\{c\}$; and $\psi(b_0)$ as pseudo-Anosov on $S_2-\{d\}$, in this case we construct surface bundles that satisfy condition (**) in Remark \ref{cond**}.

  \begin{lem}\label{red-pseu}
      Suppose $a_0$ and $b_0$ generate a free product in $\pi_1(S_1)$. Let $\psi_0(a_0)$ be a reducible symplectomorphism on $S_2$, fixing the curve $c$ that is pseudo-Anosov in $S_2-\{c\}$, similarly $\psi_0(b_0)\in Symp(S_2)$ fixes $d$ and is pseudo-Anosov in $S_2-\{d\}$. Then there exists $N\in \mathbb N$, such that for any $|k|>N$, and $\psi:\pi_1(S_1)\to Symp(S_2)$ where $\psi(a_0)=\psi_0^k(a_0),\psi(b_0)=\psi^k_0(b_0)$, then $M=S_1\ltimes_{\psi}S_2$ is a manifold that satisfies condition (**). 
  \end{lem}

  \begin{proof}[Proof of Lemma \ref{red-pseu}]
       We consider the induced action of $\psi_0(a),\psi_0(b),c,d$ on the boundary $\partial \mathbb H$ of the universal cover $\mathbb H$ of $S_2$ and use the Ping-Pong lemma to show that liftings $\alpha,\beta$ of $a,b$ generate a free subgroup in $\langle\pi_1(T_1),\pi_1(T_2)\rangle$ such that its words are all in unique reduced form.
       
       We denote by $\rho: \pi_1(T_1)*\pi_1(T_2)\to \mathrm{Homeo}(\partial \mathbb H)$ where $\rho(\alpha),\rho(\beta)$ are induced by the liftings of $\psi_0(a)$ and $\psi_0(b)$ to $\mathbb H$, and $\rho(c)$ and $\rho(d)$ are induced by the deck transformation on $\mathbb H$. Since $\psi_0(a)$ is pseudo-Anosov on $S_2-\{c\}$, we see that $\rho(a)$ act on $\partial \mathbb H$ as a homeomorphism that has four fixed points: two parabolic fixed points corresponding to the two endpoints of the lifting of $c$, a contracting fixed point which is the forward endpoint of the lifting of the stable foliation of $\psi_0(a)$, and an expelling fixed point which is the backward endpoint of the lifting of the unstable foliation of $\psi_0(a)$. 
       
       Let $X_a$ be any small open neighborhood of the four fixed points of $\rho(\alpha)$, then since $\psi_0(a)$ is pseudo-Anosov, we have $\rho(\alpha)^k(X_a^c)\subset X_a$ when $|k|$ is sufficiently large. (Here we denote by $X_a^c=\partial \mathbb H-X_a$.) Similarly, let $X_b$ be any small open neighborhood of the four fixed points of $\rho(\beta)$, then $\rho(\beta)^l(X_b)\subset X_b^c$ when $|l|$ is sufficiently large. Furthermore, the action of $\rho(c)$ and $\rho(d)$ on $\partial \mathbb H$ are induced by hyperbolic isometry on $M$, each having two contracting and expelling fixed points which are endpoints of the lifting of $c$ and $d$. We may choose $X_a$ and $X_b$ to be small enough such that $X_a\cap X_b=\varnothing$, and $c^k(X_b)\cap X_b=\varnothing, d^k(X_a)\cap X_a=\varnothing$ for any $k\neq 0$.
       
       We pick $N$ large enough, so that $\rho(\beta)^l(X_b^c)\subset X_b$ and $\rho(\alpha)^l(X_a^c)\subset X_a$ if $l\ge [\frac{N}{2}]$. And we now take $\bar\rho: \langle\pi_1(T_1),\pi_1(T_2)\rangle\to \mathrm{Homeo}(\partial \mathbb H)$, given by $\bar\rho(\alpha)=\rho(\alpha)^k=\psi^k(a)$, $\bar\rho(\alpha)=\rho(\beta)^k=\psi^k(b)$, $\bar\rho(c)=\rho(c)$, $\bar\rho(d)=\rho(d)$. By since $\pi_1(M)=\pi_1(S_1)\ltimes_{\psi_*}\pi_1(S_2)$, $\bar\rho$ is a group homomorphism.
       
       If the subgroup generated by $\alpha,\beta$ is not uniquely reduced in $\langle\pi_1(T_1),\pi_1(T_2)\rangle$, we may pick a shortest even word $$w_1=\alpha^{k_1}\beta^{l_1}\cdots \alpha^{k_m}\beta^{l_m}\in \pi_1(M), k_i,l_j\neq 0,$$ such that $w_1=w_2\in \langle \pi_1(T_1),\pi_1(T_2)\rangle$ and $w_1\neq w_2$ as a word in $\pi_1(T_1)*\pi_1(T_2)$. Since their projection to $\pi_1(S_1)$ must are equal, $w_2$ must be of the form $$w_2=\alpha^{k_1}c^{r_1}\beta^{l_1}d^{t_1}\cdots \alpha^{k_m}c^{r_m}\beta^{l_m}d^{t_m}$$ and since the word is shortest, we have $r_1\neq 0$ or $t_m\neq 0$. Without loss of generality we suppose $r_1\neq 0$ and consider $$w=w_1^{-1}w_2=\beta^{-l_m}\alpha^{-k_m}\cdots \beta^{-l_1}c^{r_1}\beta^{l_1}d^{t_1}\cdots \alpha^{k_m}c^{r_m}\beta^{l_m}d^{t_m}.$$ 
       
    Then we have $\bar\rho(w)(X_a)\subset X_b$. This is because $$\bar\rho(\beta^{l}d^{t})(X_a)=\rho^{k-[N/2]}(\beta^l)\rho(d^{t})\rho^{[N/2]}(\beta^{l})(X_b^c)\subset X_b$$ for any $l\neq 0, t\in \Z$, and similarly $\rho(\alpha^{k}c^{r})(X_b)\subset X_a$ for any $k\neq 0,r\in\Z$. Therefore, we have $$\bar\rho(\beta^{l_1}d^{t_1}\cdots \alpha^{k_m}c^{r_m}\beta^{l_m}d^{t_m})(X_a)\subset X_b,$$ and so $$\bar\rho(c^{r_1}\beta^{l_1}d^{t_1}\cdots \alpha^{k_m}c^{r_m}\beta^{l_m}d^{t_m})\subset X_b^c,$$ which implies, $$\bar\rho(\beta^{-l_m}\alpha^{-k_m}\cdots \beta^{-l_1}c^{r_1}\beta^{l_1}d^{t_1}\cdots \alpha^{k_m}c^{r_m}\beta^{l_m}d^{t_m})\subset X_b.$$ Therefore, $\bar\rho(w)\neq id$, which implies that $w\neq id$. In the case that $t_m\neq 0$, we simply consider $w_2w_1^{-1}$ instead and we may show that $\bar\rho(w_2w_1^{-1})(X_b)\subset X_a$. This shows that every word generated by $\alpha,\beta$ is uniquely reduced in the subgroup. 
  \end{proof}

\subsection{Luttinger surgery}\label{lutsurg}
Let $(M,\omega)$ be a $4$-dimensional symplectic manifold, let $T$ be a Lagrangian tori in $M$ with a closed, homotopically nontrivial curve $\gamma\subset T$, with a fixed co-orientation. We identify a neighborhood $N(T)$ of $T$ with a neighborhood of the zero section in $T^* T\simeq \mathbb T^2\times \R^2$, such that $\gamma$ is identified with $S^1\times\{pt\}\subset \mathbb T^2$, and the co-orientation is identified with the natural orientation of the second coordinate. 

We choose $r>0$ small enough, so that $U_r=\mathbb T^2\times [-r,r]\times [-r,r]$ is in the image of $N(T)$. Pick coordinates $(x,y,z,t)\in \mathbb T^2\times [-r,r]\times [-r,r]=\R/\Z\times \R/\Z\times [-r,r]\times [-r,r]$. We take $\phi_k:U_r-U_{\frac{r}{2}}\to U_r-U_{\frac{r}{2}}$ to be given by $\phi_k(x_1,x_2,v_1,v_2)=(x_1+k\chi(v_1),x_2,v_1,v_2)  $, where $\chi:[-r,r]\to [0,1]$ is a non-decreasing step function such that $$\chi(t)=\left\{\begin{array}{ccc}
    0 & &t<-\frac{r}{3} \\
   1  & & t>\frac{r}{3}
\end{array}\right.$$ and $\chi(t)+\chi(-t)=1$.

The Luttinger surgery of $M$ around $T,\gamma$ is defined as follows. This was introduced by Luttinger \cite{Lutt}. See also Section 2.1 of \cite{ADK}.

\begin{defi}
    The Luttinger surgery $M(T,\gamma,k)$ is the manifold $M(T,\gamma,k)=(M-U_{r/2})\cup_{\phi_k}U_r$, where $U_r$ is identified with a neighborhood of $T$. In other words,  $M(T,\gamma,k)$ is given by cutting out a neighborhood of $T$ and gluing back a copy of $U_r$ using the symplectomorphism $\phi_k$ around their boundaries.
\end{defi}

Let $\mu$ be the loop bounding $\partial U_r\simeq D\times \mathbb T^2$, then we see that $\pi_1(M(T,\gamma,k))=\pi_1(M-U_{r/2})/\langle\mu\gamma^{-k}\rangle$. 

Now back to our construction. We take a surface bundle example $M^\prime=S_1^\prime \ltimes_{\psi^\prime} S_2^\prime$ that satisfies condition (*) . We attach a handle to $S_1^\prime$ to gain a surface $S_1$ of one more genus, where $e,e^\prime$ are the new homotopy classes, so that $S_1/(e\cup e^\prime)\simeq S_1^\prime$. Similarly, attach a handle to $S_2^\prime$ to get a surface $S_2$, with homotopy classes $f,f^\prime$.
We take $M=S_1\ltimes_\psi S_2$ where $\psi(e)=\psi(e^\prime)=id$, and $\psi=\psi^\prime$ for homotopy classes in $\pi_1(S_1^\prime)$. We choose the torus $T$ to be generated by $e,f$ in $M$, and we choose $\gamma$ to be a lifting of $e^l$ in $T$, where $l\neq 0$. 

\begin{prop}
    Let $M^\prime$ be a surface bundle satisfying condition (*) in Definition \ref{cond*} and let $M$ be the surface bundle constructed from $M^\prime$ by attaching handles as given above, and let $T$ be a tori generated by curves in the attached handles and let $\gamma\subset T$ be a curve whose projection is not homotopically trivial. Then the Luttinger surgery $M(T,\gamma,k)$ for any $k\in \Z$ also satisfies condition (*) in Definition \ref{cond*} .
\end{prop} 

Furthermore, if $M^\prime$ satisfy condition (**) in Remark \ref{cond**}, then our argument also shows that $M(T,\gamma,k)$ satisfies condition (**).
\begin{proof}
Let $\pi: M\to S_1$ be the projection to the base. Since $e$ can be made  disjoint from $\pi(T_1\cup T_2)$, we still have Lagrangian tori $T_1$, $T_2$ in $M(T,\gamma,k)$, transversely intersecting at one point in $M-U_r$. 

Let $ T_1,T_2\subset M^\prime$ be tori satisfying condition (*), then we claim that they still generate a free product in $\pi_1(M(T,\gamma,k))$. Indeed, by our construction, $M^\prime$ can be viewed as a topological quotient of both $M$ and $M(T,\gamma,k)$. Therefore, $\pi_1(M^\prime)$ is a quotient of $\pi_1(M(T,\gamma,k))$, this shows that if $T_1$ and $T_2$ generate a free product in $\pi_1(M^\prime)$, then they still generate a free product in $\pi_1(M(T,\gamma,k))$.

We next discuss the symplectic form after the surgery and show that $M(T,\gamma,k)$ is aspherical and atoroidal.

It was proved in Proposition 2.2 of \cite{ADK} that the construction is well-defined symplectically, i.e. $M(T,\gamma,k)$ admit a natural symplectic form $\tilde{\omega}$, given by the symplectic forms on $M-U_r$ and $U_r$, whose isotopy class does not depend on the choices in the construction. 

 The manifold $M(T,\gamma,k)$ is aspherical, in fact, after the surgery we still have $\pi_2(M(T,\gamma,k))=0$. This is because the universal over of $M(T,\gamma,k)$ can be expressed by $(\tilde M-\tilde U_r)\cup_{\tilde\phi_k} \tilde U_r$ where $\tilde M$ is the universal cover of $M$, which is a disk bundle over a disk; $\tilde U_r$ is the lifting of $U_r$ to the universal cover of $M$, and is topologically copies of the tubular neighborhood of 2-dimensitonal subspace in $D\times D$; and $\tilde\phi_k$ is the lifting of $\phi_k$.  By Mayer Vietoris sequence, we have $H_2(\tilde M-\tilde U_r)\cup_{\tilde\phi_k} \tilde U_r)=0$. So by Hurewicz theorem, we have $\pi_2(M(T,\gamma,k))=0$.
 
 Now let $a,b$ be liftings of $a_0$ and $b_0$ in the construction of $M^\prime$. We now show that for any loop $\beta\in \langle a,b\rangle$, $M(T,\gamma,k)$ is $\beta$-atoroidal. Let $L\subset M(T,\gamma,k)$ be an immersed torus, generated by $\beta$ and $\beta_1$. Then the homotopy classes of $\beta$ and $\beta_1$ are commuting in $M(T,\gamma,k)$. By construction of $M(T,\gamma,k)$, we have a homomorphism $p: \pi_1(M(T,\gamma,k))\to \pi_1(S_1)$ generated by $p([\gamma])=[\pi(\gamma)]$ if $\gamma\subset M-U_r$ and $p([\mu])=[c]^{-k}$. Then since $\beta$ and $\beta_1$ commutes, we see that $p(\beta)$ and $p(\beta_1)$ commutes in $\pi_1(S_1)$, which forces $p(\beta_1)\in \langle p(\beta) \rangle$. This shows that $\beta_1$ can is homotopic to a curve outside $U_{2r}$, and therefore, $L$ is homotopic to a torus outside $U_{2r}$ in $M(T,\gamma,k)$. Since $M$ is symplectically atoroidal, and the symplectic form and the first Chern class of $M$ and $M(T,\gamma,k)$ agrees outside $U_{2r}$, we see that $M(T,\gamma,k)$ is $\beta$-atoroidal.

\end{proof}

 \subsection{Symplectic connected sum}\label{sympsum}
Given two manifolds $M_1,M_2$ with diffeomorphic codimension-2 submanifolds $N_1\subset M_1, N_2\subset M_2$, suppose that the  Euler classes of normal bundles of $N_1$ and $N_2$ are opposite, $e(N_1)=-e(N_2)$, then there is an isomorphism $\psi$ that identifies a tubular neighborhood of $N_1$ in $M_1$ with that of $N_2$ in $M_2$, which reverses the orientation of the bundles over $N_1$ and $N_2$. Therefore, we may define the sum of $M_1$ and $M_2$ along $N_1$ and $N_2$ to be $M_1\#_\psi M_2:=M_1\cup_\psi M_2$. By Theorem 1.3 of \cite{gompf}, if $(M_1,\omega_1)$ and $(M_2,\omega_2)$ are symplectic and $N_1$ and $N_2$ are symplectomorphic symplectic submanifolds, then the construction can also be made symplectic, and there exists a symplectic form $\omega$ on $M_1\#_\psi M_2$ that agrees with $\omega_1$ and $\omega_2$ outside small neighborhoods of $N_1$ and $N_2$.

Let $M_1=S_1\ltimes_{\psi_1}S_2$ be a surface bundle given in Section \ref{surfbund}, and let $M_2=S_2\ltimes_{\psi_2}S_3$ be another surface bundle such that $\psi_2 $ has a fixed point in $S_3$. Then we may choose $N_1=\{pt\}\times S_2$ to be a fiber in $M_1$, disjoint from the tori $T_1$ and $T_2$ in $M_1$ given in condition (*). We also choose $N_2$ to be the section $S_2\times \{pt\}$ on the fixed point of $\psi_2$. We also choose the monodromy $\psi_2$ so that the Euler class $e(N_2)=0$. Then we may consider $M=M_1\#_\psi M_2$ where $\psi$ is the symplectomorphism that identify a tubular neighborhood of $N_1$ with that of $N_2$. 

\begin{prop}
    Let $M_1=S_1\ltimes_{\psi_1}S_2$ and $M_2=S_2\ltimes_{\psi_2}S_3$ be two surface bundles as constructed above, where $M_1$ satisfy condition (*), and $M_2$ has a section of Euler class zero. Then their symplectic sum $M$ also satisfies condition (*), with Lagrangian tori $T_1$ and $T_2$ given by those in $M_1$. 
\end{prop}

Furthermore, if $M_1$ satisfies condition (**), then our argument shows that $M$ also satsifies condition (**).

\begin{proof}

We first compute the fundamental group of $M$. By Van Kampen theorem, we have $\pi_1(M)=\pi_1(M_1)*_{\pi_1(S_2)}\pi_1(M_2)=(\pi_1(S_1)\ltimes_{\psi_1}\pi_1(S_2)))*_{\pi_1(S_2)}(\pi_1(S_2)\ltimes_{\psi_2}\pi_1(S_3))$. This shows that there exists a monomorphism $\pi_1(M_1)\to \pi_1(M)$. Therefore, if $T_1$ and $T_2$ generate a free product in $\pi_1(M_1)$ then they also generate a free product in $\pi_1(M)$.

We next show that $M$ is aspherical. Indeed, the universal cover of $M $ is contractible, as it is the topological sums of of $D\times D$ along $D$, so we have $\pi_2(M)=0$.

Let $\beta$ be any curve in the class generated by $a\subset T_1,b\subset T_2$, we now show that $M$ is $\beta$-atoroidal. Suppose a torus $T\subset M$ is generated by $\beta$ and $\beta_1$. Since $\pi_1(M)=(\pi_1(S_1)\ltimes_{\psi_1}\pi_1(S_2)))*_{\pi_1(S_2)}(\pi_1(S_2)\ltimes_{\psi_2}\pi_1(S_3))$, and $\beta\in (\pi_1(S_1)\ltimes_{\psi_1}\pi_1(S_2))*_{\pi_1(S_2)}(\pi_1(S_2)\times\{id\})$, we see that $\beta_1$ can also be expressed as an element in $(\pi_1(S_1)\ltimes_{\psi_1}\pi_1(S_2))*_{\pi_1(S_2)}(\pi_1(S_2)\times\{id\})$. Therefore, $\beta_1$ is homotopic to a curve in $M$ that lies in $M_1-N_1$. Since $M_1$ is symplectically atoroidal, this shows that $M$ is $\beta$-atoroidal.
\end{proof}
 
\section{Preliminaries}\label{prelim}

In this section, we provide some preliminaries on plumbing space, Floer theory and the relation between Floer information and Hofer metric.
\subsection{Plumbing of cotangent bundles}
We first give a definition of the plumbing of cotangent bundles.
\begin{defi}[Plumbing]\label{plumb}
Given two manifolds $L_1$ and $L_2$ with points $p_1 \in L_1$ and $p_2 \in L_2$ and fixed Riemannian metrics on $L_1$ and $L_2$, let $U_i \subset L_i$ be neighborhoods of $p_i$. Suppose there exist local isometries $\psi_i: U_i \to B(2) \subset \mathbb{R}^n$ (the ball of radius $2$) such that $\psi_i(p_i) = 0$ and $\psi_i(U_i) = B(1)$ (the ball of radius $1$). These induce symplectomorphisms $T^*\psi_i$ from the unit disk bundles $D^*U_i$ to $D^n \times D^n \subset D^*\mathbb{R}^n \simeq \mathbb{R}^{2n}$.

Let $J: T^*\mathbb{R}^n \to T^*\mathbb{R}^n$ be the standard linear symplectomorphism given by $J(x_i) = dx_i$ and $J(dx_i) = -x_i$ for $i=1,\dots,n$. We obtain a symplectomorphism from $D^*U_1$ to $D^*U_2$ defined by $\psi = (T^*\psi_2)^{-1} \circ J \circ T^*\psi_1$.

The resulting \emph{plumbing space} $T^*L_1 \sharp T^*L_2$ is defined as the completion of the \emph{plumbing domain}
\[
P_\psi(L_1,L_2) := D^*L_1 \cup_\psi D^*L_2 / \{x \sim \psi(x) \text{ for all } x \in D^*U_1\}.
\]
\end{defi}

In this paper, for a manifold $M$ satisfying condition (*), by Weinstein's Lagrangian neighborhood theorem, a neighborhood of $T_1 \cup T_2$ in $M$ is symplectomorphic to a neighborhood of the zero sections in the plumbing domain $P_\psi(T_1,T_2)$. We shall only consider Hamiltonian diffeomorphisms supported in this small neighborhood. Therefore, we shall not distinguish between the plumbing space and the plumbing domain in this paper.

\subsection{Some Floer theory}\label{floer}
This paper utilizes (filtered) Floer homology for a specific class of orbits. Let $\mathcal{L} M$ denote the free loop space of $M$.

Let $(M,\omega)$ be a symplectically aspherical manifold, i.e., $\int_{S^2} \rho^*\omega = 0$ and $\int_{S^2} \rho^*c_1 = 0$ for any smooth map $\rho: S^2 \to M$, where $c_1$ is the first Chern class of $TM$. Fix a free, non-trivial homotopy class $\alpha \in \pi_0(\mathcal{L} M)$ with a representative $\gamma_0 \in \alpha$. We say that $M$ is $\alpha$-atoroidal if for any smooth map $\rho: \mathbb{T}^2 = S^1 \times S^1 \to M$ such that $\rho(S^1 \times \{t\}) \in \alpha$ for all $t \in S^1$, we have
\[
\int_{\mathbb{T}^2} \rho^*\omega = \int_{\mathbb{T}^2} \rho^*c_1 = 0.
\]

Let $H: S^1 \times M \to \mathbb{R}$ be a compactly supported Hamiltonian. The Floer homology of $H$ in the orbit class $\alpha$ is constructed as follows.
\begin{itemize}
    \item \textbf{Action Functional:} For a loop $\gamma: S^1 \to M$ in the class $\alpha$, define the \emph{action functional} as
\[
\mathcal{A}_H(\gamma) = \int_0^1 H_t(\gamma(t))\,dt - \int_{\bar{\gamma}} \omega,
\]
where $\bar{\gamma}: S^1 \times [0,1] \to M$ is a smooth homotopy from $\gamma_0$ to $\gamma$. By the $\alpha$-atoroidal assumption, the action is independent of the choice of $\bar{\gamma}$.

\item \textbf{Floer Chain Complex:} The critical points of $\mathcal{A}_H$ correspond to 1-periodic orbits of the Hamiltonian flow $\phi_H^t$ in the class $\alpha$. Denote the set of these orbits by $\mathcal{P}_\alpha(H)$. Since $\alpha$ is non-trivial and $H$ is compactly supported, the periodic orbits lie in the compact support of $H$.

The \emph{Floer chain complex} $\mathrm{CF}_*(H)_\alpha$ is the vector space over $\mathbb{Z}/2\mathbb{Z}$ generated by the elements of $\mathcal{P}_\alpha(H)$.

\item \textbf{Differential:} Fix a generic, time-dependent, $\omega$-compatible almost complex structure $J(t,x)$ on $M$. For two orbits $x(t), y(t) \in \mathcal{P}_\alpha(H)$, let $\mathcal{M}(x, y)$ denote the space of solutions $u: \mathbb{R} \times S^1 \to M$ to the Floer equation:
\[
\partial_s u + J(t, u)(\partial_t u - X_H(t, u)) = 0,
\]
with asymptotic conditions $u(s, t) \to x(t)$ as $s \to -\infty$ and $u(s, t) \to y(t)$ as $s \to +\infty$. Since the equation is translation-invariant in $s$, consider the moduli space $\bar{\mathcal{M}}(x, y) = \mathcal{M}(x, y) / \mathbb{R}$. For $H$ with compact support and $\alpha$ non-trivial, $\bar{\mathcal{M}}(x, y)$ is a compact manifold (possibly with boundary) for generic $J$.

The complex $\mathrm{CF}_*(H)_\alpha$ is graded by the \emph{Conley-Zehnder index} $i_{\mathrm{CZ}}(\cdot)$, which is well-defined by the aspherical and atoroidal assumptions. We adopt the convention that the Conley-Zehnder index agrees with the Morse index for small perturbations of the zero Hamiltonian (see Section \ref{precz} for details). A standard result \cite{CZ84} shows that $\dim \bar{\mathcal{M}}(x, y) = i_{\mathrm{CZ}}(x) - i_{\mathrm{CZ}}(y) - 1$. In particular, when $i_{\mathrm{CZ}}(y) = i_{\mathrm{CZ}}(x) - 1$, the moduli space $\bar{\mathcal{M}}(x, y)$ is a finite set of points.

The differential is defined by counting these elements:
\[
\partial(x) = \sum_{\substack{y \in \mathcal{P}_\alpha(H) \\ i_{\mathrm{CZ}}(y) = i_{\mathrm{CZ}}(x) - 1}} \left( \#_{\mathbb{Z}/2\mathbb{Z}} \, \bar{\mathcal{M}}(x, y) \right) \cdot y.
\]

\item \textbf{Floer Homology Group:} It is well known that $\partial \circ \partial = 0$ for generic $J$, so we define the associated homology $\mathrm{HF}_*(H)_\alpha$. In fact, $\mathrm{HF}_*(H)_\alpha$ is independent of $H$ and can be shown to be trivial by considering a small perturbation of $H=0$.
\end{itemize}

To extract non-trivial information, we use \emph{filtered Floer homology}. For $a \in \mathbb{R}$, define the filtered chain complex $\mathrm{CF}_*(H)_\alpha^{(-\infty, a)}$ as the subcomplex generated by orbits $x$ with $\mathcal{A}_H(x) < a$. Standard energy-action estimates show that $\mathcal{A}_H(\partial x) \leq \mathcal{A}_H(x)$ for any $x \in \mathcal{P}_\alpha(H)$, so $\partial$ restricts to this subcomplex. Thus, we define the filtered Floer homology $\mathrm{HF}_*(H, J)_\alpha^{(-\infty, a)}$.

More generally, for $a, b \in \mathbb{R} \cup \{\pm\infty\}$ with $a < b$, define the quotient complex
\[
\mathrm{CF}_*(H)_\alpha^{(a, b)} = \mathrm{CF}_*(H)_\alpha^{(-\infty, b)} / \mathrm{CF}_*(H)_\alpha^{(-\infty, a)},
\]
with homology denoted $\mathrm{HF}_*(H, J)_\alpha^{(a, b)}$.

We define \emph{comparison maps} $j_d: \mathrm{HF}_*(H, J)_\alpha^{(a, b)} \to \mathrm{HF}_*(H, J)_\alpha^{(a+d, b+d)}$ for shifting the action window, induced by the natural inclusion of chain complexes.

For different choices of $(H, J)$ generating the same $\phi \in \mathrm{Ham}(M,\omega)$, there is a canonical isomorphism between $\mathrm{HF}_*(H, J)_\alpha^{(a, b)}$ and $\mathrm{HF}_*(H', J')_\alpha^{(a, b)}$ that is equivariant with respect to the maps $j_d$ (see, e.g., Proposition 2.5 and Remark 2.9 of \cite{PS}). Therefore, we denote this filtered Floer homology by $\mathrm{HF}_*(\phi)_\alpha^{(a, b)}$.

If a Hamiltonian diffeomorphism $\psi$ commutes with $\phi$, then it induces a map on the set of periodic orbits of $\phi$ that preserves both the homotopy class $\alpha$ and the action functional. Consequently, $\psi$ induces a well-defined map on the filtered Floer homology:
\[
\psi_*: \mathrm{HF}_*(\phi)_\alpha^{(a, b)} \to \mathrm{HF}_*(\phi)_\alpha^{(a, b)}.
\]

\subsubsection{Conley-Zehnder index}\label{precz}
Let $\gamma$ be a 1-periodic orbit of $H$ in the class $\alpha$. The Conley-Zehnder index $i_{\mathrm{CZ}}(\gamma)$ is an integer associated with the linearized flow of $\phi_H^t$ along $\gamma$.

Choose a symplectic trivialization $\Phi: \gamma^*(TM) \to S^1 \times (\mathbb{R}^{2n}, \omega_0)$ along $\gamma$. The linearized flow defines a path of symplectic matrices $d\phi_H^t(\gamma(0))$. We define the Conley-Zehnder index of $\gamma$ as the Maslov index of this path relative to the starting point. The index is independent of the choice of trivialization by the aspherical and $\alpha$-atoroidal assumptions.

A path of symplectic matrices $\Gamma(t):[0,1]\to \mathrm{Sp}(2n)$ is \emph{non-degenerate} if $\det(\Gamma(1)-I_{2n}) \neq 0$, and \emph{degenerate} otherwise.

For the purposes of this paper, particularly for the index calculations in Section \ref{cz}, we require a definition that extends to degenerate paths. We use the Robbin-Salamon index \cite{RobbinSalamon}, which generalizes the Conley-Zehnder index to degenerate paths and assigns a half-integer value. Explicitly, the Conley-Zehnder index for a possibly degenerate orbit is defined as $n$ minus the Robbin-Salamon index for a $2n$-dimensional symplectic manifold.

A fundamental property of the Conley-Zehnder index is its relation with the dimension of moduli spaces: for $x, y \in \mathcal{P}_\alpha(H)$, we have $\dim \bar{\mathcal{M}}(x, y) = i_{\mathrm{CZ}}(x) - i_{\mathrm{CZ}}(y) - 1$.

\subsection{Boundary depth $\beta_\alpha(\phi)$}
Let $\alpha$ be a non-trivial homotopy class of free loops in $M$, and let $H$ be a compactly supported Hamiltonian. The \emph{boundary depth} of $H$, defined in \cite{Usher}, is the maximal action difference between an element in the image of the differential and a chain that bounds it. Precisely,
\[
\beta_{\alpha}(H) = \sup_{x \in \mathrm{Im}\,\partial} \inf_{y \in \partial^{-1}(x)} \left( \mathcal{A}_H(y) - \mathcal{A}_H(x) \right).
\]
Equivalently, $\beta_{\alpha}(H)$ is the length of the longest finite-length bar in the persistence barcode of the filtered Floer complex:
\[
\beta_{\alpha}(H) = \sup \left\{ d \in \mathbb{R} : \exists I \subset \mathbb{R} \text{ such that } 0 \neq j_d: \mathrm{HF}_*(H)_{\alpha}^I \to \mathrm{HF}_*(H)_{\alpha}^{I+d} \right\}.
\]

The second definition shows that $\beta_\alpha(H)$ depends only on the time-1 map $\phi = \phi_H^1$. Therefore, we define the boundary depth for the Hamiltonian diffeomorphism itself, denoted $\beta_\alpha(\phi)$.

Furthermore, the boundary depth is a Lipschitz function with respect to the Hofer norm. (While \cite{Usher} focuses on closed manifolds, the arguments extend to our setting of non-compact manifolds of finite geometry since the Hamiltonians have compact support and the Floer theory is well-defined.)

\begin{prop}\label{bounddepth}
    Let $(M,\omega)$ be an aspherical and atoroidal symplectic manifold. For any compactly supported Hamiltonian diffeomorphism $\phi \in \mathrm{Ham}(M,\omega)$ and any non-trivial homotopy class $\alpha$, we have
    \[
    \beta_\alpha(\phi) \le \|\phi\|_H.
    \]
\end{prop}

\subsection{Loop rotation operator and $\mathbb{Z}_k$ spectral spread}
In this section, we introduce the Floer-theoretic invariant used to bound the distance from a Hamiltonian diffeomorphism to the set of $k^{\text{th}}$ powers.

Let $H: S^1 \times M \to \mathbb{R}$ be a Hamiltonian. Define the $k$-fold iteration of $H$ by $H^{(k)}(t, x) = k H(kt, x)$. A standard calculation shows that $\phi_{H^{(k)}}^1 = (\phi_H^1)^k$.

The following proposition, which is a corollary of Theorem 4.22 in \cite{PS}, provides the key estimate.

\begin{prop}\label{dishamk}
    Let $H: S^1 \times M \to \mathbb{R}$ be a Hamiltonian on $M$. Suppose $H^{(k)}(t,x) = kH(kt,x)$ is non-degenerate and $\alpha \in \pi_0(\mathcal{L} M)$ is a primitive element (i.e. it is not a power of another element). Let $k$ be a prime number. For $r \in \mathbb{Z}$, suppose that there exists exactly $k$ many generators $p_i,1\le i\le k$ of $\mathrm{CF}_*(H^{(k)})_\alpha$ with index $r$, and they satisfy
    \[
    \min_{q \in \partial^{-1}(p_i)} \left( \mathcal{A}_{H^{(k)}}(q) - \mathcal{A}_{H^{(k)}}(p_i) \right) \ge D > 0.
    \]
    Then the Hofer distance is bounded below by $d_H(\phi_H, \mathrm{Ham}^k(M)) \ge \frac{1}{4k}D$.
\end{prop}

\begin{proof}
    We consider the loop rotation operator
    \[
    R_k: \mathcal{L}_\alpha M \to \mathcal{L}_\alpha M, \quad R_k(x)(t) = x(t + 1/k),
    \]
    which induces $\Z_k$-action on the filtered Floer complex $\mathrm{CF}_*(H^{(k)})_\alpha$. By \cite[Lemma 3.1]{PS}, this map agrees with the action of $(\phi_H)_*$ on the filtered Floer homology.

    Since $\alpha$ is primitive, for a generator $x$, its orbit $R_k^i(x):1\le i\le k$ under the $\mathbb{Z}_k$-action generated by $R_k$ are all distinct, therefore, by our assumption the generators of $\mathrm{CF}_*(H^{(k)})_\alpha$ with index $r$ are exactly the set $\{R_k^i(x)\}$. 

    This shows that for any unitary root $\xi\in \Z_k$, there exists exactly one element of index $r$ in the eigenspace of $\xi$ under the $\Z_k$ action of $\mathrm{CF}_*(H^{(k)})_\alpha$, which does not vanish on an interval of length $D$.

    Therefore, by Theorem 4.22 of \cite{PS}, we conclude that $d_H(\phi_H, \mathrm{Ham}^k(M)) \ge \frac{1}{4k}(D-2\epsilon)$ for any $\epsilon>0$, let $\epsilon\to 0$ yeilds the desired result.
\end{proof}

\section{Generalized linked twist map on plumbing of tori}\label{ltm}

Now we construct the Hamiltonians that we shall prove to be robustly non-autonomous, and generate a free group in the asymptotic cones of $\mathrm{Ham}(M,\omega)$. They shall be given by composition of time-change of geodesic flows on the two tori $T_1$ and $T_2$.

Let $\tau^t$ be a time change of the geodesic flow on $T^* \mathbb{T}^n$, given by $$\tau^t(v,x)=(v,x+t\rho(|v|) v),$$ where 
\[
\rho(r)= \left\{\begin{array}{cc}
  1   & r\le \epsilon/2 \\
   \frac{2(\epsilon-r)}{\epsilon}  & \epsilon/2<r\le \epsilon \\
   0 & r>\epsilon
\end{array}\right.,
\]
and $0< \epsilon <0.01$. For $\delta < 0.0001$, let $\tau_{\delta}^t(v,x) = (v, x + t\rho_\delta(|v|) v)$ be the $\delta$-smoothing of $\tau^t$, where $\rho_{\delta}(r)$ is a $\delta$-smoothing of $\rho$ as given in Figure \ref{figrho}.

\begin{figure}
    \centering
    \includegraphics[width=0.8\linewidth]{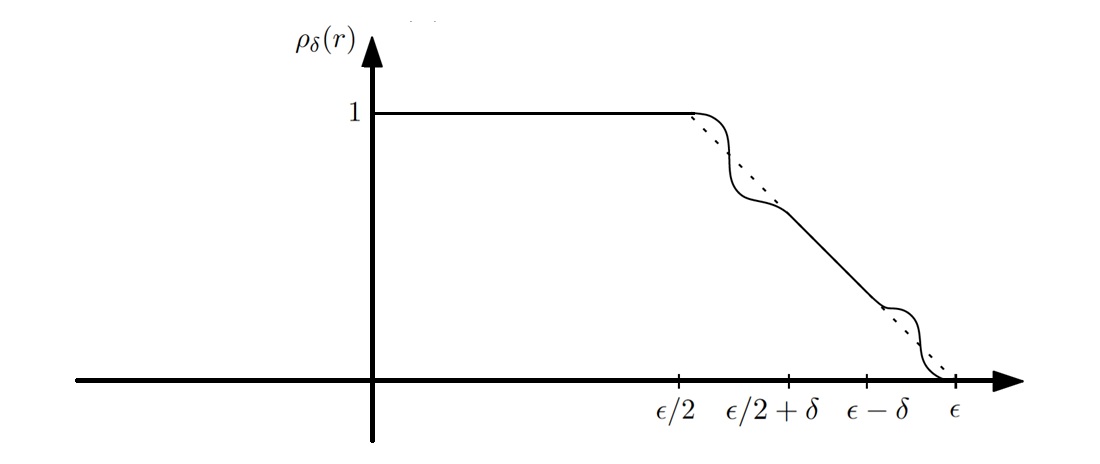}
    \caption{$\delta$-smoothing of $\rho(r)$}
    \label{figrho}
\end{figure}

The Hamiltonian path $\tau_\delta^t$ is generated by a $\delta$-smoothing of the Hamiltonian function $H(v,x)=h(|v|)$, where 
\begin{equation}\label{funch}
    h(r)= \left\{\begin{array}{cc} \frac{r^2}{2}-\frac{7}{24}\epsilon^2   & r\le \epsilon/2 \\  r^2-\frac{2r^3}{3\epsilon}-\frac{\epsilon^2}{3} & \epsilon/2<r\le \epsilon \\   0 & r>\epsilon\end{array}\right.
\end{equation}

Let $M$ be a symplectic manifold satisfying condition (*) with embedded Lagrangian tori $T_1$ and $T_2\subset M$. By Weinstein's theorem, a tubular neighborhood of $T_1$ and $T_2$ is symplectomorphic to the plumbing of the cotangent spaces of two tori $T^*_\epsilon\mathbb{T}^n \cup_\psi T^*_\epsilon \mathbb{T}^n$, where $T^*_\epsilon \mathbb{T}^n\subset T^* \mathbb{T}^n$ consists of cotangent vectors of length $\le \epsilon$, using the standard metric on $\mathbb{T}^n$. For any word in the free group $w=a^{k_1}b^{k_2}\cdots a^{k_{2m-1}}b^{k_{2m}}\in \mathbb{F}_2 = \langle a, b\rangle$, we define 
\begin{equation}\label{taunw}
    \tau(N,w)=\tau_{1,\frac{1}{N^2}}^{k_1N}\circ  \tau_{2,\frac{1}{N^2}}^{k_2N}\circ\cdots \circ \tau_{1,\frac{1}{N^2}}^{k_{2m-1}N}\circ  \tau_{2,\frac{1}{N^2}}^{k_{2m}N},
\end{equation}  
where $\tau_{1,\delta}^t$ and $\tau_{2,\delta}^t$ are the flow $\tau_{\delta}^t$ on the two copies of $T^*\mathbb{T}^n$. 

Then $\tau(N,w)$ is a Hamiltonian diffeomorphism on $M$, and we take the following explicit Hamiltonian isotopy of $\tau(N,w)$, which is the concatenation of $2m$ paths of Hamiltonian flows:
\begin{equation}\label{hampath}
\tau_{2,\frac{1}{N^2}}^{k_{2m}Nt} \# \tau_{1,\frac{1}{N^2}}^{k_{2m-1}Nt}\circ  \tau_{2,\frac{1}{N^2}}^{k_{2m}N} \#\cdots \#\tau_{1,\frac{1}{N^2}}^{k_1Nt}\circ  \tau_{2,\frac{1}{N^2}}^{k_2N}\circ\cdots \circ \tau_{1,\frac{1}{N^2}}^{k_{2m-1}N}\circ \tau_{2,\frac{1}{N^2}}^{k_{2m}N}.
\end{equation} 
The corresponding Hamiltonian is the sum of the $2m$ terms, corresponding to the intermediate paths. % REMARK: The notation $\#$ for concatenation of paths is standard, but you may want to briefly define it.

In this section, we prove the following bound on the action differences between fixed points of $\tau(N,w)$.

\begin{prop}\label{deltaA}
    Let $w=a^{k_1}b^{k_2}\cdots a^{k_{2m-1}}b^{k_{2m}}\in \mathbb{F}_2=\langle a,b\rangle$ be a word with $k_j\neq 0$, $1\le j\le 2m$. Let $M$ satisfy condition (*) and let $\tau(N,w)$ be the Hamiltonian given in Equation \ref{taunw}. Then there exist constants $C, \delta > 0$, depending only on $w$ and $\epsilon$, such that for $N > C$, there exists a free homotopy class $\gamma(N,w) \in \pi_0(\mathcal{L} M)$ and an index $r \in \mathbb{N}$ satisfying 
    \[
    \min_{q \in \partial^{-1}(p)} \left( \mathcal{A}_{\tau(N,w)}(q) - \mathcal{A}_{\tau(N,w)}(p) \right) \ge \delta N
    \]
    for any fixed point $p$ of $\tau(N,w)$ of index $r$.

    Furthermore, if $w=(a^{k_1}b^{k_2})^m$, then $\gamma(N,w)$ can be chosen to be primitive, and there exist exactly $m$ fixed points of $\tau(N,w)$ of index $r$ in the class $\gamma(N,w)$.
\end{prop}

\subsection{Proof of the main theorems}
We now use the Floer information in Proposition \ref{deltaA} to prove Theorems \ref{nokthroot} and \ref{F2incone}.

\begin{proof}[Proof of Theorem \ref{nokthroot}]
    For a prime number $k \ge 2$, $k \in \mathbb{N}$, take $f_N = \tau(N, ab)$ for $ab \in \mathbb{F}_2 = \langle a, b\rangle$. Then the $k$-periodic points of $f_N$ are exactly the fixed points of $\tau(N, (ab)^k)$. By Proposition \ref{deltaA}, there exists an index $r$ and a primitive orbit class $\gamma$ such that there exists exactly $k$ fixed points of $\tau(N, (ab)^k)$ of index $r$ and the action difference between any fixed point $p$ of index $r$ and any $q \in \partial^{-1}(p)$ is at least $\delta N$.

    Therefore, by Proposition \ref{dishamk}, we have $d_H(f_N, \mathrm{Ham}^k(M)) \ge \frac{1}{4k}\delta N$, which implies that $d_H(f_N, \mathrm{Ham}^k(M)) \to \infty$ as $N \to \infty$.

    If $k$ is not a prime integer, then $\mathrm{Ham}^k(M)\subset \mathrm{Ham}^{k^\prime}(M)$ for some prime number $k^\prime$, therefore, we still have $d_H(f_N, \mathrm{Ham}^k(M)) \to \infty$ as $N \to \infty$.
\end{proof}

\begin{proof}[Proof of Theorem \ref{F2incone}]
    Consider the embedding $\mathbb{F}_2 \to \mathrm{Cone}_\mathcal{U}(\mathrm{Ham}(M), d_H)$ defined by $w \mapsto [(\tau(N,w))_{N \in \mathbb{N}}]$, where $\tau(N,w)$ is given in Equation \ref{taunw}. This defines a group homomorphism because $\tau(N, ww') = \tau(N,w) \circ \tau(N,w')$ by construction.

    We claim that $$\liminf_{N \to \infty} \frac{\|\tau(N,w)\|_H}{N} > 0$$ for any nontrivial $w \in \mathbb{F}_2$. Then $\lim_{\mathcal{U}} \frac{\|\tau(N,w)\|_H}{N} > 0$ for any nontrivial $w$, which shows the homomorphism is injective.

    Now prove the claim that $\liminf_{N \to \infty} \frac{\|\tau(N,w)\|_H}{N} > 0$ for any non-trivial word $w \in \mathbb{F}_2$.

    For any $w \in \mathbb{F}_2$ that is not conjugate to $a^k$ or $b^k$, it is conjugate to a word $w' = a^{k_1}b^{k_2}\cdots a^{k_{2m-1}}b^{k_{2m}}$ with $k_i \neq 0$. By Proposition \ref{deltaA}, the boundary depth satisfies 
    \[
    \beta_{\gamma(N,w')}(\tau(N,w')) \ge \delta N.
    \]
    Therefore, by Proposition \ref{bounddepth}, 
    \[
    \|\tau(N,w)\|_H = \|\tau(N,w')\|_H \ge \beta_{\gamma(N,w')}(\tau(N,w')) \ge \delta N
    \]
    for sufficiently large $N$.

    For $w \in \mathbb{F}_2$ conjugate to $a^k$ or $b^k$, we use the triangle inequality and bi-invariance of the Hofer norm. For example, for $w = a^k$:
    \[
    2\|\tau_1^{kN}\|_H = \|\tau_1^{kN}\|_H + \|\tau_2^N \tau_1^{kN} \tau_2^{-N}\|_H \ge \|\tau_2^N \tau_1^{kN} \tau_2^{-N} \tau_1^{-kN}\|_H.
    \]
    The right-hand side is $\|\tau(N, w')\|_H$ for a commutator word $w'$ which is nontrivial, so by the previous case, $\|\tau(N, w')\|_H \ge \delta N$. Thus, $\|\tau_1^{kN}\|_H \ge \delta N/2$.  A similar inequality also holds for $\tau_2^{kN}$. Therefore, for $w\in \mathbb{F}_2$ that is conjugate to $\alpha^k$ or $\beta^k$, there also exists $\delta>0$ such that $||\tau(N,w)||_H\ge \delta N$ for sufficiently large $N$.
\end{proof}

To prove Proposition \ref{deltaA}, we shall describe all fixed points of $\tau(N,w)$ in a certain class $\gamma$ and estimate their actions and Conley-Zehnder indices.

\subsection{Fixed points in a certain class}

Suppose $p = (v_0, x_0) \in [-\frac{1}{2}, \frac{1}{2}]^n \times [-\frac{1}{2}, \frac{1}{2}]^n$ is a fixed point of $\tau(N,w)$, whose corresponding orbit is in the class $\gamma = \alpha_1 * \beta_1 * \alpha_2 * \beta_2 * \cdots * \alpha_m * \beta_m \in \pi_1(M, *)$, where the base point is the plumbing point of the two tori. Here we take $\alpha_i\in \pi_1(T_1), \beta_i\in \pi_1(T_2)$, which we also view as elements of $\Z^n$. (Note this is different from the assumption that the orbit class is $\gamma(N,w)$ in the loop space, where $\gamma$ only needs to be in the conjugacy class of $\gamma(N,w)$.) % REMARK: The $a_i, \beta_i$ here are elements of $\mathbb{Z}^n$ representing homology classes. Please clarify this notation.

Suppose $0 \neq \alpha_i, \beta_i \in \mathbb{Z}^n$. Let 
\[
(v_i, x_i) = \tau_{1,\frac{1}{N^2}}^{k_{2m-2i-1}N} \circ \tau_{2,\frac{1}{N^2}}^{k_{2m-2i}N} \circ \cdots \circ \tau_{1,\frac{1}{N^2}}^{k_{2m-1}N} \circ \tau_{2,\frac{1}{N^2}}^{k_{2m}N} (v_0, x_0),
\]
then by our construction, we have 
\begin{equation}\label{perpts}
    \begin{array}{c}
        x_{i+1} = x_i + Nk_{2m-2i}\rho_{\frac{1}{N^2}}(|v_i|)v_i - \alpha_i \in [-\frac{1}{2}, \frac{1}{2}]^n, \\
        v_{i+1} = v_i - Nk_{2m-2i-1}\rho_{\frac{1}{N^2}}(|x_{i+1}|)x_{i+1} - \beta_i \in [-\frac{1}{2}, \frac{1}{2}]^n, \\
        x_m = x_0, \quad v_m = v_0.
    \end{array}
\end{equation}
Recall that we assume $k_i \neq 0$ for all $1 \le i \le 2m$.

We shall show that for certain choices of $\gamma$ and large enough $N$, the Hamiltonian diffeomorphism $\tau(N,w)$ has exactly $2^{2m}$ fixed points in the class $\gamma$, and we shall estimate their actions.

For $1 \le i \le 2m$, choose $\alpha_i \in \langle a \rangle$, $\beta_i \in \langle b \rangle$ to be of length $\frac{N\epsilon}{4} \le |\alpha_i|, |\beta_i| \le \frac{N \epsilon}{3}$. Suppose $N > \frac{10 \max_i |k_i|}{\epsilon}$.

By Equation \ref{perpts}, we have for $0 \le i \le m$ that 
\begin{equation}\label{asympx}
    x_i \rho_{\frac{1}{N^2}}(|x_i|) = \frac{-\beta_{i-1} - v_i + v_{i-1}}{k_{2m-2i-1}N} = \frac{-\beta_{i-1}}{k_{2m-2i-1}N} + \frac{v_{i-1} - v_i}{k_{2m-2i-1}N}.
\end{equation}

This implies that $\frac{1}{N^2} < |x_i| < \epsilon - \frac{1}{N^2}$ and $||x_i| - \frac{\epsilon}{2}| < \frac{1}{N^2}$. Similarly, we have $|v_i|, |v_{i+1}|$ lie in the region where $\rho = \rho_{\frac{1}{N^2}}$. Combining with Equation \ref{asympx}, this shows that $x_i \in B_i^- \cup B_i^+$, where 
\[
B_i^* = B\left( r_i^* \cdot \frac{\beta_{i-1}}{|\beta_{i-1}|}, \frac{\max_j |k_j|}{N} \right)
\]
and $r_i^+, r_i^- \in \mathbb{R}$ are roots of the equation 
\[
r \rho(|r|) = -\frac{|\beta_{i-1}|}{k_{2m-2i-1}N}
\]
satisfying $\frac{\epsilon}{4 \max_i |k_i|} < |r_i^-| < \frac{\epsilon}{3} < \frac{\epsilon}{2} < |r_i^+| < \epsilon$. Similarly, we may show that $v_i \in C_i^- \cup C_i^+$, where 
\[
C_i^* = B\left( s_i^* \cdot \frac{\alpha_i}{|\alpha_i|}, \frac{\max_j |k_j|}{N} \right),
\]
and $s_i^-, s_i^+ \in \mathbb{R}$ are roots of the equation 
\[
s \rho(|s|) = \frac{|\alpha_i|}{k_{2m-2i}N}
\]
satisfying $\frac{\epsilon}{4 \max_i |k_i|} < |s_i^-| < \frac{\epsilon}{3} < \frac{\epsilon}{2} < |s_i^+| < \epsilon$.

\begin{lem}\label{persig}
    Suppose $k_i \neq 0$, and $N$ is sufficiently large. Then for any set of signs $\sigma_i, \xi_i \in \{+, -\}$, there exists exactly one periodic point satisfying Equation \ref{perpts} with $x_i \in B_i^{\sigma_i}$ and $v_i \in C_i^{\xi_i}$.
\end{lem}

\begin{proof}
    Let $V \subset \mathbb{R}^{2mn}$ be the closure of $C_1^{\xi_1} \times B_1^{\sigma_1} \times \cdots \times C_m^{\xi_m} \times B_m^{\sigma_m}$. Consider the map $f: V \to \mathbb{R}^{2mn}$ with coordinates 
    \[
    (a_0, b_0; a_1, b_1; \cdots; a_{m-1}, b_{m-1}), \quad a_i, b_i \in \mathbb{R}^n,
    \]
    given by 
    $$
        f(a_0, b_0; \cdots; a_{m-1}, b_{m-1}) $$$$= ( a_{m-1} + Nk_2 \rho(|b_{m-1}|) b_{m-1} - a_{m-1},
         b_{m-1} - Nk_{1} \rho(|a_m|) a_m - b_m;$$$$
         a_0 + Nk_{2m} \rho(|b_0|) b_0 - a_1,  b_0 - Nk_{2m-1} \rho(|a_0|) a_0 - b_1; $$$$
         \cdots; $$$$
         a_{m-2} + Nk_4 \rho(|b_{m-2}|) b_{m-2} - a_{m-2},  b_{m-2} - Nk_3 \rho(|a_{m-1}|) a_{m-1} - b_{m-2}).$$
  
    Then proving the lemma is equivalent to showing that $f$ restricted to $V$ has a unique fixed point.

    \textbf{Step 1}:For $p=(a_0,b_0;\cdots;a_{m-1},b_{m-1})\in V$, we take its norm to be $|p|_\infty=max\{|a_0|,|b_0|,\cdots,|a_{m-1}|,|b_{m-1}|\}$. We claim that $$|f(p)-f(p^\prime)|_\infty\ge (\frac{N}{5 \max_j|k_j|}-1)|p-p^\prime|_\infty$$ for any $p,p^\prime \in V$. In particular, this shows that $f$ is injective, and any fixed point must be unique. 

    To prove this claim, we take $p = (a_0, b_0; \cdots; a_{m-1}, b_{m-1})$ and $p' = (a_0', b_0'; \cdots; a_{m-1}', b_{m-1}') \in V$. Assume without loss of generality that 
    \[
    \max_j |a_j - a_j'| \ge \max_j |b_j - b_j'|,
    \]
    and let $i$ be such that $|a_i - a_i'| = \max_j |a_j - a_j'|$. Then $|a_i - a_i'| \ge |b_j - b_j'|$ for all $1 \le j \le m$. By the definition of $f$, we have 
    \begin{align*}
        |b_i(f(p)) - b_i(f(p'))| &\ge N |k_{2m-2i-1}| \cdot \left| \rho(|a_i|) a_i - \rho(|a_i'|) a_i' \right| - |b_{i-1} - b_{i-1}'| \\
        &\ge \left( \frac{N}{5 \max_j |k_j|} - 1 \right) |a_i - a_i'|,
    \end{align*}
    since $a_i, a_i' \in C_i^{\xi_i}$. Therefore, 
    \[
    |f(p) - f(p')|_\infty \ge |b_i(f(p)) - b_i(f(p'))| \ge \left( \frac{N}{5 \max_j |k_j|} - 1 \right) |p - p'|_\infty.
    \]
    The case where $\max_j |a_j - a_j'| \le \max_j |b_j - b_j'|$ is similar.

    \textbf{Step 2}: we show that $V\subset f(V)$. By our definition of $V$ and $f$, $f$ is a smooth diffeomorphism from $V$ to $f(V)$. For $p=(a_1,b_1,...,a_m,b_m)\in V$, we have $$Df(p)=\begin{pmatrix}
    0&-L_1(a_1)&I&0&\cdots&0\\
     0&0 &K_2(b_1)&I&\ddots&\vdots\\
     \vdots&\cdots &\ddots&\ddots&\ddots&0\\
       0&\cdots &\cdots&0&K_m(b_{m-1})&I\\
            I&0 &\cdots&0&0&-L_m(a_{m})\\
                   K_1(b_m)&I&0 &\cdots&0&0 \end{pmatrix}$$ where $$L_i=k_{2m-2i+1}N(\rho(|a_i|)I_n+\frac{\rho^\prime(|a_i|)}{|a_i|}a_ia_i^T)$$$$=\left\{\begin{array}{cc}
   k_{2m-2i+1}NI_n  & \xi_i=-  \\
   k_{2m-2i+1}N \cdot P\cdot diag(\frac{2(\epsilon-2|a_i|)}{\epsilon},\rho(|a_i|),\cdots,\rho(|a_i|))P^{-1}& \xi_i=+
\end{array}\right.$$ where $P$ is an invertible matrix with first column $a_i$; and $$K_i(x)=k_{2i}N(\rho(|x|)I_n+\frac{\rho^\prime(|x|)}{|x|}xx^t.$$ 

This shows that for $v=(u_1,v_1, ...,u_m,v_m)\in T_pV$, we have $$|Df(p)(v)|^2=\sum|K_iu_i+v_i|^2+|-L_iv_i+u_{i+1}|^2$$$$\ge \sum(|K_iu_i|^2+|L_iv_i|^2)-2\sum(|K_iu_i||v_i|+|L_iv_i||u_{i+1}|)$$$$\ge\sum \frac{N^2}{25 \max_j |k_j|^2}(|u_i|^2+|v_i|^2)-2\sum \max_j |k_j|\cdot N (|u_i||v_i|+|v_i||u_{i+1}|)$$
$$\ge (\frac{N^2}{25 \max_j |k_j|^2}-4 \max_j|k_j|\cdot N ) |v|^2.$$

Therefore, for $N>200  \max_j|k_j|^3$, we have $||Df_p||\ge \frac{N}{10 \max_j |k_j|}$.

Take $p_0=(s^{\xi_1}_1\cdot \frac{\alpha_1}{|\alpha_1|}, r_1^{\sigma_1}\cdot \frac{\beta_1}{|\beta_1|},...,s^{\xi_m}_m\cdot \frac{\alpha_m}{|\alpha_m|}, r_m^{\sigma_m}\cdot \frac{\beta_m}{|\beta_m|}).$ Then we have $f(p_0)=(s^{\xi_m}_m\cdot \frac{\alpha_m}{|\alpha_m|}, r_m^{\sigma_m}\cdot \frac{\beta_m}{|\beta_m|},s^{\xi_1}_1\cdot \frac{\alpha_1}{|\alpha_1|}, r_1^{\sigma_1}\cdot \frac{\beta_1}{|\beta_1|},...,s^{\xi_{m-1}}_{m-1}\cdot \frac{\alpha_{m-1}}{|\alpha_{m-1}|}, r_{m-1}^{\sigma_{m-1}}\cdot \frac{\beta_{m-1}}{|\beta_{m-1}|})\in (-\epsilon,\epsilon)^{2mn}.$

Therefore, we see that $f(V)$ contains a ball of radius at least $\frac{1}{10}$, centered at $f(p_0)$. Combining with the assumption $\epsilon<0.01$, this shows that $V\subset (-\epsilon,\epsilon)^{2mn}\subset f(V)$. 

    \textbf{Step 3}: From Step 1 and 2, we see that $f^{-1}: V\to V$ is a diffeomorphism that is contracting under the norm defined in Step 1. Therefore by Banach fixed point theorem, $f^{-1}$ has a unique fixed point in $V$. This shows that $f$ also has a unique fixed point in $V$.
\end{proof}

\subsection{Proof of Proposition \ref{prop:denseorbits}}
Now we have the characterization of fixed points of $\tau(N,w)$ given in Lemma \ref{persig}, we may give a proof of Proposition \ref{prop:denseorbits}.

\begin{proof}[Proof of Proposition \ref{prop:denseorbits}]
    Let $f_n=\tau(n,ab)$ as in the proof of Theorem \ref{nokthroot}, then we see that $k$-periodic points of $f_n$ are exactly fixed points of $\tau(n,(ab)^k)$.
    By Lemma \ref{persig}, for any family of curves $\alpha_1,\alpha_2,\ldots, \alpha_k;\beta_1,\beta_2, \ldots,\beta_k$ with $\frac{n\epsilon}{4}\le|\alpha_i|,|\beta_i|\le \frac{n\epsilon}{3}$, there exists $2^{2k}$ many $k$-periodic point of $f_n$ in class $\gamma=\alpha_1*\beta_1*\ldots *\alpha_k*\beta_k$. Therefore, the number of $k$-periodic points of $f_n$ grows exponentially in $k$.

    Pick $U=\{(v,x)\in T^*T_1: \frac{\epsilon}{4}< |v|,|x|< \frac{\epsilon}{3}\}$. For each $n$, we take a family curves $\alpha_1(n),\alpha_2(n),\ldots,\alpha_{k(n)}(n);\beta_1(n),\beta_2(n),\ldots,\beta_{k(n)}(n)$, so that the set $\{(\frac{\alpha_{i}(n)}{n},-\frac{\beta_{i-1}(n)}{n}):1\le i\le k(n)\}$ consists of all pairs of rational vectors with denominator $n$ in $U$. Take $p_n$ to be the fixed point of $\tau(n,(ab)^{k(n)})$ in the class $\alpha_1*\beta_1*\ldots *\alpha_{k(n)}*\beta_{k(n)}$ of signature $\sigma_i=-,\xi_i=-$ for all $1\le i\le k(n)$. Then by Lemma \ref{persig}, such fixed point $p_n$ exists, and $f_n^k(p_n)\in B(\frac{\alpha_{i}(n)}{n},\frac{1}{n})\times  B(-\frac{\beta_{i-1}(n)}{n},\frac{1}{n})$. 

   By our construction, for any open set $V\subset U$, there exists $N>0$, such that for any $n>N$, there exists $1\le i\le k(n)$ such that $B(\frac{\alpha_{i}(n)}{n},\frac{1}{n})\times  B(-\frac{\beta_{i-1}(n)}{n},\frac{1}{n})\subset V$. Therefore, we have $f_n^i(p_n)\in V$, and so $\mathcal{O}(f_n,p_n)\cap V\neq \varnothing$.
\end{proof}

\subsection{Action estimate}
We now estimate the action of fixed points in the orbit class $\gamma$.

\begin{prop}\label{pact}
    Suppose $k_i \neq 0$ and $N$ is sufficiently large. Take the Hamiltonian flow generating $\tau(N,w)$ as in Equation \ref{hampath}. Suppose $p = (v_0, x_0)$ is a fixed point of $\tau(N,w)$ satisfying Equation \ref{perpts}, such that $x_i \in B_i^{\sigma_i}$ and $v_i \in C_i^{\xi_i}$ for a set of signs $\sigma_i, \xi_i \in \{+, -\}$. Then 
    \[
    \mathcal{A}(p) = \sum_{i=0}^{m-1} \left( \mathcal{A}_{2i} + \mathcal{A}_{2i+1} \right),
    \]
    where 
    \[
    \mathcal{A}_{2j} = N k_{2m-2j} \left( h(r_j^{\sigma_j}) + r_j^{\sigma_j} \frac{|\beta_j|}{N k_{2m-2j}} \right) + O(1)
    \]
    and 
    \[
    \mathcal{A}_{2j+1} = N k_{2m-2j-1} \left( h(s_j^{\xi_j}) + s_j^{\xi_j} \frac{|\alpha_j|}{N k_{2m-2j}} \right) + O(1),
    \]
    where $h$ is the function given in Equation \ref{funch}.
\end{prop}

\begin{proof}
    Since the Hamiltonian flow is a concatenation of the $2m$ paths, the action of $p$ is the sum of the actions on the intermediate paths. We have 
    \[
    \mathcal{A}(p) = \sum_{i=0}^{m-1} \left( \mathcal{A}_{2i} + \mathcal{A}_{2i+1} \right),
    \]
    where 
    \begin{align*}
        \mathcal{A}_{2j} &= \int_0^1 H(\tau^{k_{2m-2j} N t}(x_j, v_j)) \, dt + \int_{\{\tau^{k_{2m-2j} N t}(x_j, v_j)\}} \lambda - \int_{\beta_j} \lambda \\
        &= N k_{2m-2j} \left( h(r_j^{\sigma_j}) + r_j^{\sigma_j} \frac{|\beta_j|}{N k_{2m-2j}} \right) + O(1),
    \end{align*}
    as $N \to \infty$, since $v_j \in B\left( r_j^{\sigma_j} \cdot \frac{\beta_j}{|\beta_j|}, \frac{1}{N} \right)$. Similarly,
    \begin{align*}
        \mathcal{A}_{2j+1} &= \int_0^1 H(\tau^{k_{2m-2j-1} N t}(v_{j+1}, -x_j)) \, dt - \int_{\{\tau^{k_{2m-2j-1} N t}(v_{j+1}, -x_j)\}} \lambda \\
        &= N k_{2m-2j-1} \left( h(s_j^{\xi_j}) + s_j^{\xi_j} \frac{|\alpha_j|}{N k_{2m-2j}} \right) + O(1).
    \end{align*}

\end{proof}

%For any $k$, we may pick intervals $I_i,J_i\subset (\frac{1}{4},\frac{1}{3})$ such that for any choice of  $\mu_i\in I_i,\nu_i\in J_i$, the sum $\sum_{i=0}^{k-1} (h(r_i^{\sigma_i})+r_i^{\sigma_i}\nu_i+h(s_i^{\xi_i})+s^{\xi_i}_i\mu_{i-1})$ is different for different sequence of signs $\{\sigma_i\}, \{\xi_i\}$, with pair-wise difference larger than some $\delta>0$.

%This shows that as $N\to \infty$, there exists choices of $\alpha_i$, $\beta_i$, for which the action of possible periodic points in class $\gamma$ are at least different by $\delta N+O(1)$. 
\section{Conley–Zehnder index}\label{cz}
In this section, we compute the Conley–Zehnder index of the fixed points of $\tau(N,w)$ in class $\gamma$ and give a proof of Proposition \ref{deltaA}.

We pick the framing on the reference loop $\gamma$ to be $\frac{\partial}{\partial x_i}, \frac{\partial}{\partial v_i}$ along $\alpha_i$ and $\frac{\partial}{\partial v_i}, -\frac{\partial}{\partial x_i}$ along $\beta_i$. 

We shall prove the following result.

\begin{prop}\label{czindex}
    Let $(v_0,x_0)$ be a fixed point of the linked twist map $\tau$ as given in Section \ref{ltm}, and let $(v_i,x_i)$ be the intermediate points of $(x_0,v_0)$ as given in Equation \ref{perpts}. Suppose $x_i \in B_i^{\sigma_i}$ and $v_i \in C_i^{\xi_i}$ for $\sigma_i, \xi_i \in \{+,-\}$. With abuse of notations, we also take $\sigma_i,\xi_i=\pm1$ according to its sign. Then the Conley–Zehnder index of $(v_0,x_0)$ is 
    \[
    i_{CZ}(\Gamma(t)) = n - \frac{1}{2} \sum_{i=1}^m \left( \mathrm{sign}(k_{2i-1})(\sigma_i + 1 - n) + \mathrm{sign}(k_{2i})(\xi_i + 1 - n) \right).
    \]
    % REMARK: The signs $\sigma_i, \xi_i$ are treated as $\pm 1$ in the formula. Please verify this convention.
\end{prop}

We prove this by computing the Robbin–Salamon index of the intermediate paths and their concatenation.

Since the Hamiltonian flow is a concatenation of $2m$ paths, the differential along the Hamiltonian flow of a fixed point $(x_0,v_0)$ with intermediate points $(x_i,v_i)$ is given by $\Gamma(t) = \Gamma_{2m}(t) \# \Gamma_{2m-1}(t) \cdot \Gamma_{2m}(1) \# \cdots \# \Gamma_1(t) \cdot \Gamma_2(1) \cdots \Gamma_{2m}(1)$, where 
\[
\Gamma_{2i-1}(t) = d \left( \tau^{k_{2i-1} Nt}_{\alpha, \frac{1}{N^2}} \right) = \begin{pmatrix} I_n & L_i(v_{m-i}) t \\ 0 & I_n \end{pmatrix}
\]
and
\[
\Gamma_{2i}(t) = d \left( \tau^{k_{2i} Nt}_{\beta, \frac{1}{N^2}} \right) = \begin{pmatrix} I_n & 0 \\ -K_i(x_{m-i}) t & I_n \end{pmatrix}.
\]
Here the matrices are given by
\begin{multline}\label{Leig}
L_i(v) = k_{2i-1} N \left( \rho(|v|) I_n + \frac{\rho'(|v|)}{|v|} v v^t \right) \\
= \left\{ \begin{array}{cc}
   k_{2i-1} N I_n & v \in C_i^- \\
   k_{2i-1} N \cdot P \cdot \mathrm{diag}\left( \frac{2(\epsilon - 2|v|)}{\epsilon}, \rho(|v|), \cdots, \rho(|v|) \right) P^{-1} & v \in C_i^+
\end{array} \right.
\end{multline}
where $P$ is an invertible matrix with first column $v$; and
\begin{multline}\label{Keig}
K_i(x) = k_{2i} N \left( \rho(|x|) I_n + \frac{\rho'(|x|)}{|x|} x x^t \right) \\
= \left\{ \begin{array}{cc}
   k_{2i} N I_n & x \in B_i^- \\
   k_{2i} N \cdot Q \cdot \mathrm{diag}\left( \frac{2(\epsilon - 2|x|)}{\epsilon}, \rho(|x|), \cdots, \rho(|x|) \right) Q^{-1} & x \in B_i^+
\end{array} \right.
\end{multline}
where $Q$ is an invertible matrix with first column $x$.

We compute the Robbin–Salamon index of the intermediate paths using formulas for concatenation of paths given in \cite{Gutt} and \cite{DDP}. In particular, we prove the following lemmas.

\begin{lem}\label{nondegRS}
    If $k_i \neq 0$ and $N$ is sufficiently large, then the Conley–Zehnder index of $\Gamma(t)$ is equal to the sum of the Robbin–Salamon indices of the intermediate paths, i.e.,
    \[
    i_{RS}(\Gamma(t)) = \sum_{j=1}^m i_{RS} \left( \Gamma_{2j}(t) \# \Gamma_{2j-1}(t) \cdot \Gamma_{2j}(1) \right).
    \]
\end{lem}

\begin{lem}\label{degRS}
    If $k_i \neq 0$ and $N$ is sufficiently large, then for any $1 \le j \le m$,
    \[
    i_{RS} \left( \Gamma_{2j}(t) \# \Gamma_{2j-1}(t) \cdot \Gamma_{2j}(1) \right) = i_{RS}(\Gamma_{2j}(t)) + i_{RS}(\Gamma_{2j-1}(t)).
    \]
\end{lem}

We also use the following theorem which computes the Robbin–Salamon index of a symplectic shear.

\begin{thm}[Proposition 23, \cite{Gutt}]\label{RSshear}
    Let $B(t)$ be a family of symmetric matrices. Then the Robbin–Salamon index of the symplectic shear $\phi(t) = \begin{pmatrix} I & B(t) \\ 0 & I \end{pmatrix}$ is
    \[
    i_{RS}(\phi(t)) = \frac{1}{2} \mathrm{sign}\, B(0) - \frac{1}{2} \mathrm{sign}\, B(1).
    \]
\end{thm}

These lemmas allow us to compute the Conley–Zehnder index of $\Gamma(t)$.

\begin{proof}[Proof of Proposition \ref{czindex}]
    By Lemmas \ref{nondegRS} and \ref{degRS}, we have
    \[
    i_{RS}(\Gamma(t)) = \sum_{i=1}^m \left( i_{RS}(\Gamma_{2i}(t)) + i_{RS}(\Gamma_{2i-1}(t)) \right).
    \]

    By Theorem \ref{RSshear}, we have
    \[
    i_{RS}(\Gamma_{2i-1}(t)) = -\frac{1}{2} \mathrm{sign}\, L_i(v_{m-i}) = -\frac{1}{2} \mathrm{sign}(k_{2i-1}) \cdot (n - 1 - \xi_i)
    \]
    and
    \[
    i_{RS}(\Gamma_{2i}(t)) = i_{RS}(J \Gamma_{2i}(t) J^{-1}) = -\frac{1}{2} \mathrm{sign}\, K_i(x_{m-i}) = -\frac{1}{2} \mathrm{sign}(k_{2i}) \cdot (n - 1 - \sigma_i).
    \]

    Therefore, we have
    \[
    i_{RS}(\Gamma(t)) = \sum_{i=1}^m \frac{1}{2} \left( \mathrm{sign}(k_{2i-1})(\xi_i + 1 - n) + \mathrm{sign}(k_{2i})(\sigma_i + 1 - n) \right).
    \]

    This shows
    $$i_{CZ}(\Gamma(t)) = n - i_{RS}(\Gamma(t)) $$ $$= n - \frac{1}{2} \sum_{i=1}^m \left( \mathrm{sign}(k_{2i-1})(\sigma_i + 1 - n) + \mathrm{sign}(k_{2i})(\xi_i + 1 - n) \right).
    .$$
\end{proof}

\subsection{Proof of Proposition \ref{deltaA}}
We are now ready to prove Proposition \ref{deltaA}.

\begin{proof}[Proof of Proposition \ref{deltaA}]
    For sufficiently large $N$, pick $\gamma(N,w)$ to be the conjugacy class of $a_1 * b_1 * \cdots * a_m * b_m \in \pi_1(M)$ such that $\frac{N\epsilon}{4} \le |a_i|, |b_i| \le \frac{N \epsilon}{3}$ and $a_i$, $b_j$ are powers of $a$ and $b$ respectively. Since any word in $a, b$ is uniquely reduced in the group $\langle \pi_1(T_1), \pi_1(T_2) \rangle$, any fixed point of $\tau(N,w)$ with orbit class $\gamma_N$ must satisfy Equation \ref{perpts} with $\alpha_i = a_{i+e}$, $\beta_i = b_{i+e}$ for some $e \in \{0, 1, \cdots, m-1\}$. 
    
    In the general case where $w$ is not a power, pick $a_1 = a_2 = \cdots = a_m$, $b_1 = b_2 = \cdots = b_m$. Then by Lemma \ref{persig}, there are $2^{2m}$ fixed points of $\tau(N,w)$, whose Conley–Zehnder indices are given by Proposition \ref{czindex}. 
    
    From Proposition \ref{czindex}, there exists a unique fixed point $p = (v_0, x_0)$ of $\tau(N,w)$ with minimum Conley–Zehnder index $r$, such that its intermediate points $(v_i, x_i)$ given in Equation \ref{perpts} satisfy $x_i \in B_i^{-\mathrm{sign}(k_{2i-1})}$ and $v_i \in C_i^{-\mathrm{sign}(k_{2i})}$. Let $q \in \partial^{-1}(p)$. Then the Conley–Zehnder index of $q$ is $r-1$, and the signs $\sigma_i$ and $\xi_i$ for $p$ and $q$ are the same except for one entry, which we assume without loss of generality to be $\sigma_j$.

    By Proposition \ref{pact}, for any $q \in \partial^{-1}(p)$,
    \[
    \left| \mathcal{A}_{\tau(N,w)}(p) - \mathcal{A}_{\tau(N,w)}(q) \right| = \left| \mathcal{A}_{2j}(p) - \mathcal{A}_{2j}(q) \right| + O(1) \ge \frac{N \epsilon^2}{10 \max_j |k_j|} + O(1)
    \]
    for sufficiently large $N$. 

    In the case where $w = (a^{k_1} b^{k_2})^m$, denote $w_1 = a^{k_1} b^{k_2}$. Pick $a_1 * b_1 * \cdots * a_m * b_m \in \pi_1(M)$ to have no symmetry, i.e., $a_1 * b_1 * \cdots * a_m * b_m \neq a_{1+e} * b_{1+e} * \cdots * a_{m+e} * b_{m+e}$ for any $e = 1, 2, \cdots, m-1$. By Lemma \ref{persig}, there are $2^{2m}$ sets of fixed points of $\tau(N,w)$ of the form 
    \[
    \{ p, \tau(N,w_1) p, \tau(N,w_1^2) p, \cdots, \tau(N,w_1^{m-1}) p \}.
    \]
    By definition, the fixed points in the orbit of $\tau(N,w_1)$ have the same index and action. Then by Proposition \ref{czindex} the same argument as above shows that there is a unique set of fixed points with minimum index, with action difference larger than $\frac{N \epsilon^2}{10 \max_j |k_j|} + O(1)$.

    This proves the claim.
\end{proof}

\subsection{Proof of Lemma \ref{nondegRS}}
In this section, we prove Lemma \ref{nondegRS} using the following formula for the Robbin–Salamon index of concatenation of non-degenerate paths.

Let $P$ be a symplectic matrix with $(I - P)$ invertible. The Cayley transform of $P$ is defined as the symmetric matrix
\[
M_P = \frac{1}{2} J (I + P) (I - P)^{-1}.
\]

The following theorem, a combination of Corollary 3.5 and Lemma 3.3 of \cite{DDP}, allows us to compute the Robbin–Salamon index for a concatenation of non-degenerate paths of symplectic matrices.

\begin{thm}[Corollary 3.5 and Lemma 3.3, \cite{DDP}]\label{nondegconcat}
    Let $A(t), B(t) \in \mathrm{Sp}_{2n}(\mathbb{R})$ be two non-degenerate paths of symplectic matrices with $A(0) = B(0) = I$. Suppose $A(1) - I$ and $B(1) - I$ are both invertible. Then the Robbin–Salamon index of their concatenation is
    \[
    i_{RS}( \{A(t)\} \# \{B(t) \circ A(1)\} ) = i_{RS}(A(t)) + i_{RS}(B(t)) - \frac{1}{2} \left( \mathrm{sign}(M_{A(1)} + M_{B(1)}) \right).
    \]
    Here, the signature of a symmetric matrix is the number of positive eigenvalues minus the number of negative eigenvalues.
\end{thm}

Now we prove Lemma \ref{nondegRS} by applying the theorem above.

\begin{proof}[Proof of Lemma \ref{nondegRS}]
    Denote $\Phi_j(t) = \Gamma_{2j}(t) \# \Gamma_{2j-1}(t) \cdot \Gamma_{2j}(1)$ for $1 \le j \le m$. Then
    \[
    \Phi_i(1) = \begin{pmatrix} I_n - L_i(v_{m-i}) K_i(x_{m-i}) & L_i(v_{m-i}) \\ -K_i(x_{m-i}) & I_n \end{pmatrix}.
    \]

    \textbf{Step 1}: Let $P = \Phi_i(1) \cdot \Phi_{i+1}(1) \cdots \Phi_m(1)$. We claim that for sufficiently large $N$, the matrix $I - P$ is invertible and $\mathrm{sign}(M_P) = 0$. Similarly, $I - \Phi_i(1)$ is invertible and $\mathrm{sign}(M_{\Phi_i(1)}) = 0$.

    Let $\bar{L}_i = \frac{L_i(v_{m-i})}{N}$, $\bar{K}_i = \frac{K_i(x_{m-i})}{N}$, and $C_i = \bar{L}_i \bar{K}_i \cdots \bar{L}_m \bar{K}_m$. Then the entries of $\bar{L}_i, \bar{K}_i, C_i$ lie in a compact subset independent of $N$. We have
    \[
    P = (-1)^{m-i} \begin{pmatrix} -N^{m-i+2} C_i & N^{m-i+1} C_i \bar{K}_m^{-1} \\ -N^{m-i+1} \bar{L}_i^{-1} C_i & N^{m-i} \bar{L}_i^{-1} C_i \bar{K}_m^{-1} \end{pmatrix} \cdot \left(1 + O\left(\frac{1}{N}\right)\right).
    \]
    Therefore, we have
    \[
    (I - P)^{-1} = \begin{pmatrix}
        \frac{1}{N^2} \bar{K}_m^{-1} \bar{L}_i^{-1} + \frac{(-1)^{m-i}}{N^{m-i+2}} C_i^{-1} & -\frac{1}{N} \bar{K}_m^{-1} \\
        -\frac{1}{N} \bar{L}_i^{-1} & I_n
    \end{pmatrix} \cdot \left(1 + O\left(\frac{1}{N}\right)\right),
    \]
    where $O(1/N)$ denotes a matrix with entries bounded by $O(1/N)$.

    This shows that
    \[
    M_P = \frac{1}{2} J (I + P) (I - P)^{-1} = -\frac{1}{2} J + J (I - P)^{-1} = \begin{pmatrix} 0 & \frac{1}{2} I_n \\ \frac{1}{2} I_n & 0 \end{pmatrix} \cdot \left(1 + O\left(\frac{1}{N}\right)\right).
    \]
    This shows that $\mathrm{sign}(M_P) = 0$ for sufficiently large $N$.

    \textbf{Step 2}: We use induction to prove that the Robbin–Salamon index of the concatenation is the sum of the indices of the intermediate paths.

    Applying Theorem \ref{nondegconcat} to $A(t) = \Phi_i(t)$ and $B(t) = \Phi_{i+1}(t) \# \Phi_{i+2}(t) \# \cdots \# \Phi_m(t)$, and using Step 1, we obtain
    \[
    i_{RS}(\Phi_i(t) \# \Phi_{i+1}(t) \# \cdots \# \Phi_m(1)) = i_{RS}(\Phi_i(t)) + i_{RS}(\Phi_{i+1}(t) \# \cdots \# \Phi_m(1)).
    \]
    By induction on $i$ from $m-1$ down to $1$,
    \[
    i_{RS}(\Gamma(t)) = \sum_{i=1}^m i_{RS}(\Phi_i(t))
    \]
    for large enough $N$. This proves the lemma.
\end{proof}

\subsection{Proof of Lemma \ref{degRS}}
To compute the Robbin–Salamon index for concatenation of degenerate paths, we use a different base point.

For $P \in \mathrm{Sp}(2n)$ such that $(P - J)$ is invertible, consider the symmetric matrix
\[
C_J(P) = J (J - I) (P - J)^{-1} (P - I).
\]
The following theorem follows from Lemma 3.3 and Corollary 3.7 of \cite{DDP}.

\begin{thm}[Lemma 3.3 and Corollary 3.7, \cite{DDP}]\label{degRSconcat}
    Let $A(t), B(t) \in \mathrm{Sp}_{2n}(\mathbb{R})$ be two paths of symplectic matrices with $A(0) = B(0) = I$. Suppose $A(1) - J$ and $B(1) - J$ are both invertible. Then the Robbin–Salamon index of their concatenation is
    $$
    i_{RS}( \{A(t)\} \# \{B(t) \circ A(1)\} ) = i_{RS}(A(t)) + i_{RS}(B(t)) $$$$+ \frac{1}{2} \left( \mathrm{sign}(C_J(A(1)) - C_J(B(1))) - \mathrm{sign}(C_J(B(1))) + \mathrm{sign}(C_J(A(1))) \right).
   $$ 
    Here, the signature of a symmetric matrix is the number of positive eigenvalues minus the number of negative eigenvalues.
\end{thm}

We use this formula to show that $i_{RS}(\Gamma_{2j-1}(t) \# \Gamma_{2j}(t)) = i_{RS}(\Gamma_{2j-1}(t)) + i_{RS}(\Gamma_{2j}(t))$.

\begin{proof}[Proof of Lemma \ref{degRS}]
    For sufficiently large $N$, by Equations \ref{Leig} and \ref{Keig}, the absolute values of the eigenvalues of $K_i$ and $L_i$ are $\ge \frac{N}{10}$. A direct computation shows that $\Gamma_{2j}(1) - J$ is invertible, and
    \[
    C_J(\Gamma_{2j}(1)) = J (J - I) (\Gamma_{2j}(1) - J)^{-1} (\Gamma_{2j}(1) - I) = \begin{pmatrix} -2I + 4(2I - K_i(x_{m-i}))^{-1} & 0 \\ 0 & 0 \end{pmatrix}.
    \]
    Similarly,
    \[
    C_J(\Gamma_{2j-1}(1)) = \begin{pmatrix} 0 & 0 \\ 0 & -2I + 4(2I - L_i(v_{m-i}))^{-1} \end{pmatrix}.
    \]
    Since the eigenvalues of $K_i$ and $L_i$ have absolute value $\ge \frac{N}{10}$, for large enough $N$ we have $\mathrm{sign}(C_J(\Gamma_{2j}(1))) = -n$, $\mathrm{sign}(C_J(\Gamma_{2j-1}(1))) = -n$, and $\mathrm{sign}(C_J(\Gamma_{2j-1}(1)) - C_J(\Gamma_{2j}(1))) = 0$.

    Applying Theorem \ref{degRSconcat} to $A(t) = \Gamma_{2j-1}(t)$ and $B(t) = \Gamma_{2j}(t)$ yields
    \[
    i_{RS}(\Gamma_{2j-1}(t) \# \Gamma_{2j}(t)) = i_{RS}(\Gamma_{2j-1}(t)) + i_{RS}(\Gamma_{2j}(t)).
    \]
\end{proof}

\appendix

\bibliography{infhof}{}
\bibliographystyle{acm}

\end{document}